\documentclass{lms}
\usepackage{graphicx}
\usepackage{amssymb}
\usepackage{amsmath}
\usepackage{amsfonts}
\usepackage{amsmath,amscd}
\usepackage{hyperref}

\newtheorem{thm}{Theorem}[section]
\newtheorem{cor}[thm]{Corollary}
\newtheorem{propn}[thm]{Proposition}
\newtheorem{lem}[thm]{Lemma}
\newtheorem{thm3}{Theorem 1.3}

\newnumbered{conj}[thm]{Conjecture}
\newnumbered{rmk}[thm]{Remark}
\newnumbered{ques}[thm]{Question}
\newunnumbered{org}{Organization}
\newnumbered{defn}[thm]{Definition}
\newunnumbered{setup}{Setup}
\newunnumbered{adhyp}{Additional hypotheses}
\newunnumbered{extseif}{Extending to a Seifert surface}
\newunnumbered{exm}{Examples}
\newnumbered{ex}[thm]{Example}
\newunnumbered{refthm2}{Refining Theorem 1.2}
\newunnumbered{oddtwist}{Twists of odd order}

\title{Cosmetic two-strand twists on fibered knots}

\author{Carson Rogers}

\classno{57M25 (primary), 57M12 (secondary)}

\begin{document}
	
\maketitle

\newcommand{\m}{\mbox{ }}

\graphicspath{{./images/}}

\begin{abstract}
Let $K$ be a knot in a rational homology sphere $M$. This paper investigates the question of when modifying $K$ by adding $m>0$ half-twists to two oppositely-oriented strands, while keeping the rest of $K$ fixed, produces a knot isotopic to $K$. Such a \textit{two-strand twist} of order $m$, as we define it, is a generalized crossing change when $m$ is even and a non-coherent band surgery when $m=1$. A \textit{cosmetic} two-strand twist on $K$ is a non-nugatory one that produces an isotopic knot. We prove that fibered knots in $M$ admit no cosmetic generalized crossing changes. Further, we show that if $K$ is fibered, then a two-strand twist of odd order $m$ that is determined by a separating arc in a fiber surface for $K$ can only be cosmetic if $m=\pm 1$. 

After proving these theorems, we further investigate cosmetic two-strand twists of odd orders. Through two examples, we find that the second theorem above becomes false if `separating' is removed, and that a key technical proposition fails when the order equals 1. A closer look at an order-one example, an instance of cosmetic band surgery on the unknot, reveals it to be nearly trivial in a sense that we name \textit{weakly nugatory}. We correct the technical proposition to obtain a means of using double branched covers to show that certain band surgeries are weakly nugatory. As an application, we prove that every cosmetic band surgery on the unknot is of this type.
\end{abstract}

\section{Introduction}

Let $K$ be a knot in an oriented 3-manifold $M$ and $B$ be a 3-ball in $M$ intersecting $K$ in two properly embedded arcs. Assume that, once $K$ has been oriented somehow, $B\cap K$ is a trivial two-string tangle appearing as on the left side of Figure \ref{ntangle_rep}. That figure illustrates what we mean by making a \textit{two-strand $n$-twist} on $K$ for an integer $n>0$: this notion will be properly defined in Section 2.4. For $n<0$, a two-strand $n$-twist is likewise described by reversing all crossings on the right side of Figure \ref{ntangle_rep}. Note that the result is always a new knot $K'$ in $M$, due to the relative orientations of the two strands. We refer to the unknotted circle $c$ on the left side of the figure as the corresponding \textit{twisting circle} for $K$. 

\begin{figure}[b]
	\centering
	\includegraphics[scale=0.4]{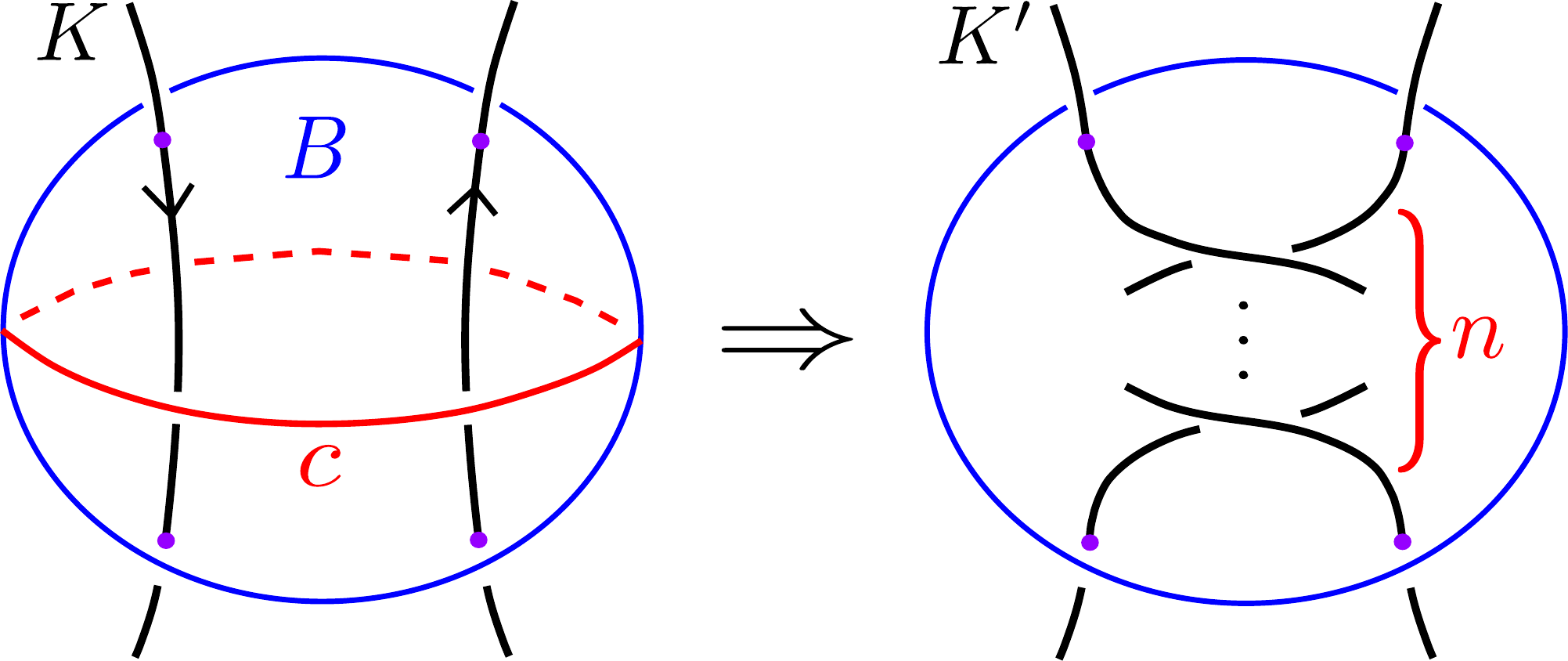}
	\caption{Performing a two-strand twist on a knot when $n>0$.}
	\label{ntangle_rep}
\end{figure}

This paper addresses the question of when a two-strand twist on a knot $K$ in a rational homology sphere produces a knot isotopic to $K$. A two-strand twist on $K$ is said to be \textit{nugatory} if the corresponding twisting circle bounds a disk embedded in $M-K$. In this case, it follows immediately that the resulting knot is isotopic to $K$. We say that a two-strand twist on $K$ is \textit{cosmetic} if it is not nugatory, but the resulting knot is still isotopic to $K$. 

Our study of this distinction is primarily motivated by two special cases. When $n$ is even, a two-strand $n$-twist is the same as a generalized crossing change of order $|n|/2$, and a two-strand $\pm 1$-twist is a type of non-coherent band surgery. The first of these cases has attracted the attention of low-dimensional topologists for some time, particularly for knots in the 3-sphere. 

\begin{conj}\label{ccc}
(Generalized cosmetic crossing conjecture) No knot in $S^3$ admits a cosmetic generalized crossing change.
\end{conj}

The case of standard crossing changes (of order 1) is Problem 1.58 of \cite{Kir78}, which has been verified in the cases of 2-bridge knots \cite{Tor99}, fibered knots \cite{Kal12}, and knots whose branched double covers are L-spaces which satisfy a particular homological condition \cite{LM17}. In fact, the first and second of these works address the generalized form of the conjecture. Obstructions to  cosmetic generalized crossing changes have also been found for genus one knots and satellite knots \cite{BFKP12} \cite{BK16}, which resolve the conjecture for some knots of these types. 

The Montesinos trick can be used to reduce Conjecture \ref{ccc} to a question about cosmetic Dehn surgery in the cyclic double cover of $S^3$ branched along the knot. This technique is utilized in \cite{Tor99} and \cite{LM17}, and as explained in Section 2.4, it extends directly to the study of cosmetic two-strand $n$-twists on null-homologous knots in rational homology spheres when $|n|\geq 2$. In Section 3, we exploit this framework to prove the following results.

\begin{thm}\label{mainthm1}
Fibered knots in rational homology spheres do not admit any cosmetic generalized crossing changes. 
\end{thm}

\begin{thm}\label{mainthm2}
Let $K$ be a fibered knot in a rational homology sphere and $c$ be a twisting circle for $K$. Suppose that: 
\begin{itemize}
\item
	some fiber surface $F$ for $K$ is disjoint from $c$, and;	

\item
	a corresponding twisting arc for $K$ lies in and separates $F$.
\end{itemize}
The two-strand $n$-twist on $K$ determined by $c$ is then cosmetic for at most one integer $n$, which must be one of -1 and 1. 
\end{thm}

The definition of `twisting arc' is given in Section 2.4. The first theorem generalizes Kalfagianni's resolution of Conjecture \ref{ccc} for fibered knots in $S^3$. Of course, in light of Theorem \ref{mainthm1}, Theorem \ref{mainthm2} only provides new information in the case that $n$ is odd. It is made more interesting by the fact that, as we will see in Section 4.1, the statement becomes false upon removing the separating hypothesis on the twisting arc.

To establish Theorem \ref{mainthm1} and Theorem \ref{mainthm2}, we begin by using the Montesinos trick as mentioned above, together with a result of Gabai \cite{Ga87}, to reduce their proofs to a unified argument. A key fact is that, in both contexts, the relevant two-strand twist may be assumed to preserve the fiber surface. The central step is Theorem \ref{corethm}, which applies a result of Ni \cite{Ni11} to describe the relationship between the double covers of $M$ branched over two fibered knots of the same genus that are related by such a two-strand $n$-twist for $|n|\geq 2$. In particular, we will find that squared monodromies for the two fibered knots must differ by a power of a Dehn twist along a simple closed curve in the fiber surface, up to conjugation in the mapping class group.

\begin{oddtwist}
The \textit{order} of a two-strand $n$-twist on a knot is defined to be $|n|$. Following the proof of the above theorems, we turn our attention to the case of odd-order two-strand twists. In Section 4.1, we will find that there are simple examples of cosmetic two-strand twists of odd orders: in fact, both the unknot and the figure-eight knot in $S^3$ admit them. These examples will be enough to reveal why there is no easy way to extend the kind of reasoning developed in Section 3 to make stronger conclusions about two-strand twists of odd orders. In particular, our example of a cosmetic two-strand $-5$-twist on the figure-eight knot will show that there is a genuine need for the hypotheses on the twisting arc in Theorem \ref{mainthm2}.

Our other example is of a cosmetic two-strand $-1$-twist on the unknot. This gives rise to such cosmetic twists on every knot in $S^3$. These can be regarded as non-trivial, purely cosmetic band surgeries in the sense of \cite{IJ18}. However, in Section 4.2, we will find that they should still be viewed as `nearly trivial' in a certain sense, for which we coin the term \textit{weakly nugatory}. 

At the same time, the existence of such a two-strand twist on the unknot reveals that the conclusion of Proposition \ref{liftprop}, which is central to our application of the Montesinos trick in Section 3, is false in the case of two-strand twists of order one. Fortunately, it turns out that not all is lost. Through Proposition \ref{liftprop2}, we will see that the Montesinos trick can still be applied to show that certain two-strand $\pm 1$-twists are weakly nugatory. As a basic application, we obtain the following result.
\end{oddtwist}

\begin{thm}\label{mainthm3}
Every cosmetic two-strand $\pm 1$-twist on the unknot is weakly nugatory.
\end{thm}

This can be viewed as saying that the unknot does not admit any truly surprising cosmetic band surgeries. Following the proof of this theorem, we discuss the difficulty in saying more about the order-one case in the context of Theorem \ref{mainthm2}. 

\begin{org}
In Section 2, we establish the background required to prove Theorem \ref{mainthm1} and Theorem \ref{mainthm2}. We then prove these theorems in Section 3. While readers with knowledge of fibered knots can postpone looking over Sections 2.1.2, 2.2, and 2.3 until after encountering the statement of Theorem \ref{corethm}, it is essential to read Section 2.4 before beginning Section 3. The case of two-strand $n$-twists for odd $n$ is treated in Section 4. After introducing the key examples in Section 4.1, we focus on the order-one case in Section 4.2, and prove Theorem \ref{mainthm3} via Proposition \ref{liftprop2}. Section 4.3 concludes the paper by posing some questions.
\end{org}

\begin{acknowledgements}
The author wishes to thank Cameron Gordon and Tao Li for helpful discussions over the course of completing this work, as well as Tye Lidman, Allison Moore, and Yi Ni for helpful email correspondence. Special gratitude is owed to Tye and Yi, whose feedback on an early version of this work draft led to substantial improvements of the main results. In particular, much of the content beyond Theorem \ref{mainthm1} may not have come into existence if not for an inspirational question of Yi's. Last, but not least, the author thanks the referee for providing helpful comments on the first draft, including a valuable suggestion to improve a central piece of terminology.
\end{acknowledgements}

\section{Preliminaries}

We will always use $M$ to denote a closed, oriented 3-manifold. Say that $M$ is a \textit{rational homology sphere} if $H_*(M,\mathbb{Q})\cong H_*(S^3,\mathbb{Q})$. If $K$ is a knot in $M$, we use $M_K$ to denote the compact exterior $\overline{M-\eta(K)}$ of $K$ in $M$, where $\eta(K)$ is a tubular neighborhood of $K$ in $M$. For basic terminology and facts from 3-manifold topology that are taken for granted here, we refer the reader to \cite{Ro76} and \cite{Sc14}.

\subsection{Fibered knots}

Let $K$ be a null-homologous knot in $M$. We say that $K$ is \textit{fibered} if $M_K$ admits the structure of a surface bundle over $S^1$ in which the boundary of each fiber surface is a longitude of $K$. A fiber then extends to a \textit{Seifert surface} for $K$, a compact, orientable surface $F$ embedded in $M$ with $\partial F=K$. We will freely blur this distinction, viewing a fiber surface $F$ as properly embedded in $M_K$ or as a surface bounded by $K$ according to convenience.

\subsubsection{Fundamental facts}

To begin, we note that every closed 3-manifold does in fact contain a fibered knot \cite{Go75}. As the focus of this article is on rational homology spheres, we note that many concrete examples of fibered knots in rational homology spheres other than $S^3$ arise as lifts of fibered knots in $S^3$ to corresponding branched covers of prime power orders. (See p.362 of \cite{Gor72} for a proof that all such branched covers have first Betti number equal to 0.) \\

We will need the following uniqueness properties of fiber surfaces. A Seifert surface $F$ for a null-homologous knot $K$ in $M$ is said to be of \textit{minimum genus} if $g(F')\geq g(F)$ for every other Seifert surface $F'$ for $K$.

\begin{propn}\label{unfib}
Let $K$ be a fibered knot in a rational homology sphere $M$, with fiber surface $F$. If $F'$ is another Seifert surface for $K$, then the following are equivalent:
\begin{description}
	\item{(a) }
	$F'$ is of minimum genus.
	
	\item{(b) }
	$F'$ is another fiber surface for $K$.
	
	\item{(c) }
	$F'$ is isotopic to $F$ (rel $K$).
\end{description}
\end{propn}
\begin{proof}
The equivalence of (a) and (b) is contained in Lemma 2.2 of \cite{Ko89}, and it is clear that (c) implies (b). To complete the proof from here, it suffices to show that (a) implies (c). Note that since $M$ is a rational homology sphere and $K$ is null-homologous, there is a unique longitude of $K$ on $\partial\eta(K)$ which bounds in $M_K$. If $F'$ is as in (a), we may therefore assume that $\partial F'=\partial F$ in $M_K$, and that $\mbox{Int}(F')\cap \mbox{Int}(F)$ consists of mutually disjoint closed curves. Since $F'$ is incompressible, one may use Corollary 3.2 of \cite{Wa68} to remove components of this intersection one-by-one via isotopies of $F'$ (rel boundary). Another application of that result then implies that $F'$ is parallel to $F$ in $M_K$.
\end{proof}

Assume that $K$ is fibered. Let $F$ be a fiber surface for $K$ and $N(F)$ be a bicollar on $F$ in $M_K$. The surface exterior $M_F=\overline{E(K)-N(F)}$ is then homeomorphic to $F\times[-1,1]$, so that $\partial M_F\cap \partial N(F)$ is identified with $F\times\{-1,1\}$. We may reconstruct the knot exterior $M_K$ from $M_F$ by gluing $F\times\{-1\}$ to $F\times\{1\}$ together via an orientation-preserving homeomorphism $\phi:F\rightarrow F$, fixing $\partial F$ pointwise, so that $(x,1)$ is identified with $(\phi(x),-1)$. We say that $\phi$ is a \textit{monodromy} for $K$, and write $M_K=F\times[-1,1]/\phi$.

\begin{propn}\label{unmon}
Let $K$ and $K'$ be fibered knots of the same genus in a rational homology sphere $M$, so that $M_K=F\times[-1,1]/\phi$ and $M_{K'}=F\times [-1,1]/\phi'$ for some compact oriented surface $F$. If $K$ is isotopic to $K'$, then:
\begin{description}
	\item{(a) }
	There is an orientation-preserving, fiber-preserving homeomorphism between $M_K$ and $M_{K'}$.
	
	\item{(b) }
	There is an orientation-preserving homeomorphism $h:F\rightarrow F$, fixing $\partial F$ pointwise, such that $\phi$ is isotopic to $h\phi' h^{-1}$. In other words, $\phi$ and $\phi'$ represent conjugate elements of the mapping class group of $F$.
\end{description}
\end{propn}

The fact that (a) holds if $K$ and $K'$ are isotopic can be viewed as a consequence of the equivalence of parts (b) and (c) of Proposition \ref{unfib}$\m$. Statements (a) and (b) are equivalent, as follows from the general theory of fiber bundles: this fact does not require $M$ to be a rational homology sphere. It can also be proven using classical 3-dimensional techniques and theorems on surface homeomorphisms: see Proposition 5.10 of \cite{BHZ14}, whose proof does not require the ambient manifold to be $S^3$. 

\subsubsection{Heegaard splittings induced by fibrations}

Let $K$ be a fibered knot in a 3-manifold $M$, so that $M_K=F\times[-1,1]/\phi$ for some compact oriented surface $F$ and  orientation-preserving homeomorphism $\phi:F\rightarrow F$ fixing $\partial F$ pointwise. Under this identification, for each $t\in[-1,1]$, let $A_t$ be an annulus embedded in $\eta(K)=\overline{M-M_K}$ such that $(F\times\{t\})\cap A_t=\partial F\times\{t\}$ and $\partial A_t=(\partial F\times\{t\})\cup K$. We can and will assume that $A_s\cap A_t=K$ when $s\neq t$ except for when $s=-1$ and $t=1$, in which case $A_s=A_t$.

Let $F_t=(F\times\{t\})\cup A_t$ for each $t$. Then $S=F_0\cup F_1$ is a closed, orientable surface which separates $M$. We may label the closures of the components of $M-S$ as $V_0$ and $V_1$, so that $V_i$ contains $F\times[i-1,i]\subset M_K$ for $i=0,1$. Further, $V_i$ is precisely the union of $F\times[i-1,i]$ with all of the annuli $A_t$ for $t\in[i-1,i]$, so $V_i$ is topologically obtained from $F\times[i-1,i]$ by thickening $(\partial F)\times[i-1,i]$ slightly into $\eta(K)$. 

Thus, both $V_0$ and $V_1$ are handlebodies (of genus $2g(F)$), so $S$ is a Heegaard surface for $M$. We say that a Heegaard splitting $M=V_0\cup_SV_1$ arising from this construction is \textit{induced by the fibration} of $M_K$. Note that, by our conventions, the knot $K$ lies on the Heegaard surface $S$, instead of in the interior of one of the handlebodies.

\subsection{Powers of Dehn twists in mapping class groups}

We adhere to the standard conventions of \cite{FM12}, in which left Dehn twists are viewed as positive. The reader is referred to that text for the basic notions around Dehn twists and mapping class groups.

In Theorem 7 of \cite{Kot04}, Kotschick uses the theory of Lefschetz fibrations of 4-manifolds to obtain a lower bound on the commutator lengths of powers of products of positive Dehn twists, within mapping class groups of closed, oriented surfaces of genus two and above. As observed by Kalfagianni \cite{Kal12}, his result immediately implies that powers of Dehn twists (positive or negative) along non-trivial curves in such surfaces cannot be commutators. We use $\mbox{MCG}(S)$ denote the mapping class group of the oriented surface $S$. If $\phi:S\rightarrow S$ is an orientation-preserving homeomorphism which fixes $\partial S$ pointwise, then we will use $[\phi]$ to denote the element of $\mbox{MCG}(S)$ represented by $\phi$.

\begin{thm}[Corollary 2.2 of \cite{Kal12}]\label{kotthm}
Suppose that $S$ is a closed oriented surface of genus $g(S)\geq 2$, and $\alpha$ is a simple closed curve in $S$. If $[T^n_{\alpha}]=ghg^{-1}h^{-1}$ for some integer $n\neq 0$ and $g,h\in\mbox{MCG}(S)$, then $\alpha$ is homotopically trivial in $S$.
\end{thm}

We will need an analogous statement for compact surfaces with one boundary component. Happily, the best possible statement in this case is an immediate consequence of the above result.

\begin{cor}\label{kotcor}
Theorem \ref{kotthm} holds true in the case that $S$ is any compact oriented surface with $\partial S\neq\emptyset$ and $\alpha$ is a simple closed curve in $\mbox{Int}(S)$.
\end{cor}
\begin{proof}
Construct a new surface $\widehat{S}$ by gluing a once-punctured torus to $S$ along each of its boundary curves, and extending the orientation on $S$ to the resulting closed surface. Suppose that $\alpha$ is a simple closed curve in $\mbox{Int}(S)$ such that $[T^n_{\alpha}]=ghg^{-1}h^{-1}$ for some $g,h\in\mbox{MCG}(S)$ and $n\neq 0$. By applying the inclusion-induced homomorphism $\mbox{MCG}(S)\rightarrow\mbox{MCG}(\widehat{S})$ to this equation, it follows that $T^n_{\alpha}$ also represents a commutator in $\mbox{MCG}(\widehat{S})$.

By Theorem \ref{kotthm}, it follows that $\alpha$ is homotopically trivial in $\widehat{S}$, which means that it bounds a disk in $\widehat{S}$. At the same time, if $\beta$ is a separating simple closed curve in $S$, then by our construction of $\widehat{S}$, any subsurface of $\widehat{S}$ bounded by $\beta$ that is not contained in $S$ must have genus $\geq 1$. Since $\alpha$ is such a curve, this means that the disk bounded by $\alpha$ must be contained in $S$. We therefore have the desired conclusion.
\end{proof}

\subsection{Dehn surgery in product manifolds}

By a \textit{product manifold}, we mean a 3-manifold of the form $F\times I$ for a compact, orientable surface $F$, where $I$ is a closed interval. For simplicity, we let $I=[0,1]$. Ni characterized the knots in product manifolds on which some non-trivial Dehn surgery returns the same product manifold \cite{Ni11}, up to a homeomorphism preserving $F\times\{0\}$ and $F\times\{1\}$. 

To state his result, let $K$ be a knot in $F\times I$ and $p:F\times I\rightarrow F$ be the natural projection. We will assume that any knot $K'$ in $F\times I$ under consideration is in general position with respect to $p$, so that $p|_{K'}$ is an embedding away from a finite set of transverse double points in its image. For an integer $n\geq 0$, say that $K$ is an \textit{$n$-crossing knot} if it can be isotoped so that this map has exactly $n$ double points, and no knot isotopic to $K$ admits fewer double points in its image under $p$. The following is a restatement of Theorem 1.1 of \cite{Ni11}.

\begin{thm}[(Ni)]\label{nithm}
Let $F$ be a compact, orientable surface and $K$ be a knot in $F\times I$. Suppose that performing Dehn surgery along $K$ with slope $\alpha$ yields a manifold homeomorphic to $F\times I$, via a homeomorphism which preserves each of $F\times\{0\}$ and $F\times\{1\}$. Then one of the following holds:
\begin{description}
	\item{(a) }
	$K$ is a 1-crossing knot, and $\alpha$ is the blackboard framing of $K$ determined by the corresponding projection onto $F$.
	
	\item{(b) }
	$K$ is a 0-crossing knot, and $\alpha=1/n$ for some nonzero integer $n$ with respect to the blackboard framing determined by $F$.
\end{description}
\end{thm}

In particular, if $\alpha$ is not an integral slope, then $K$ and $\alpha$ must be as in case (b).

\subsection{Relating two-strand twists to surgery in branched double covers}

Recall that Figure \ref{ntangle_rep} from the introduction illustrates what we mean by a two-strand $n$-twist on a knot $K$ in an oriented 3-manifold $M$. We will now be more precise. Regard $K$ as being oriented, though the choice of orientation will not matter. Define a \textit{twisting circle} for $K$ to be an embedded curve $c$ in $M_K$ bounding an embedded disk $D$ in $M$ which intersects $K$ twice transversely and zero times algebraically. We refer to $D$ as the corresponding \textit{twisting disk}.

\begin{defn}\label{maindef}
Let $B$ be a bicollar on $D$ in $M$, parametrized as $D^2\times[0,1]$, so that $D^2\times\{t\}$ intersects $K$ in two points of opposite sign for each $t\in[0,1]$. Let $T=B\cap K$. We regard the disks $D^2\times\{t\}$ as being oriented so that their positive normal vectors point in the direction of increasing $t$. For a nonzero integer $n$, a \textit{two-strand $n$-twist} on $K$ determined by $c$ is the operation of modifying the pair $(B,T)$ by the twisting homeomorphism $\tau_n:B\rightarrow B$ defined by $$\tau_n(p,t)=(e^{-itn}p,t)\m\m\mbox{ for all }p\in D^2\mbox{ and }t\in[0,1].$$ 
\end{defn}

\begin{rmk}
Note that this turns $K$ into a new knot $K'$ in $M$, though $K'$ does not inherit an orientation from $K$ when $n$ is odd. However, one can check that the isotopy class of $K'$ depends only on that of $c$ in $M_K$, and not on the particular choices of $D$ and $B$. When $n$ is even, performing the two-strand $n$-twist determined by $c$ is equivalent to performing Dehn surgery on $c$ with slope $-2/n$, which is precisely the operation of a generalized crossing change of order $|n|/2$ \cite{Kal12}.
\end{rmk}

Another bit of terminology is required. A \textit{twisting arc} for $K$ determined by a twisting circle $c$ is an arc $\gamma$ embedded in $M$ such that $\partial\gamma=K\cap\gamma$ and, for some twisting disk $D$ bounded by $c$, $\gamma$ is isotopic to an arc embedded in $D$ via an isotopy which fixes $K$ pointwise. It should be noted that, in the special case of generalized crossing changes, other authors typically refer to this as a `crossing arc,' $c$ as a `crossing circle,' and $D$ as a `crossing disk.'

For all that follows, assume that $M$ is a rational homology sphere and $K$ is a null-homologous knot in $M$. We may then consider the cyclic double cover $\Sigma(K)$ of $M$ branched over $K$. A twisting arc $\gamma$ for $K$ lifts to a simple closed curve $\widetilde{\gamma}$ in $\Sigma(K)$. By a special case of the Montesinos trick, performing a two-strand twist on $K$ corresponding to $\gamma$ changes $\Sigma(K)$ by a specific kind of Dehn surgery along $\widetilde{\gamma}$. We fully describe this correspondence below. The reader can turn to \cite{Gor09} for an exposition of the general Montesinos trick.

Let $B$ be a small 3-ball in $M$ containing $\gamma$ in its interior, so that $\partial B$ intersects $K$ transversely in exactly four points. (We may view $B$ as a small bicollar on $D$, as in Definition \ref{maindef}$\m$.) The double cover of $\overline{M-B}$ branched over $\overline{K-B}$ can be seen to equal $\overline{\Sigma(K)-\eta(\widetilde{\gamma})}$, where $\eta(\widetilde{\gamma})$ is a tubular neighborhood of $\widetilde{\gamma}$ in $\Sigma(K)$. Let $T$ be the tangle $B\cap K$, and $K'$ be the knot obtained from $K$ by performing a two-strand twist corresponding to $\gamma$.

\begin{figure}[b]
	\centering
	\includegraphics[scale=0.4]{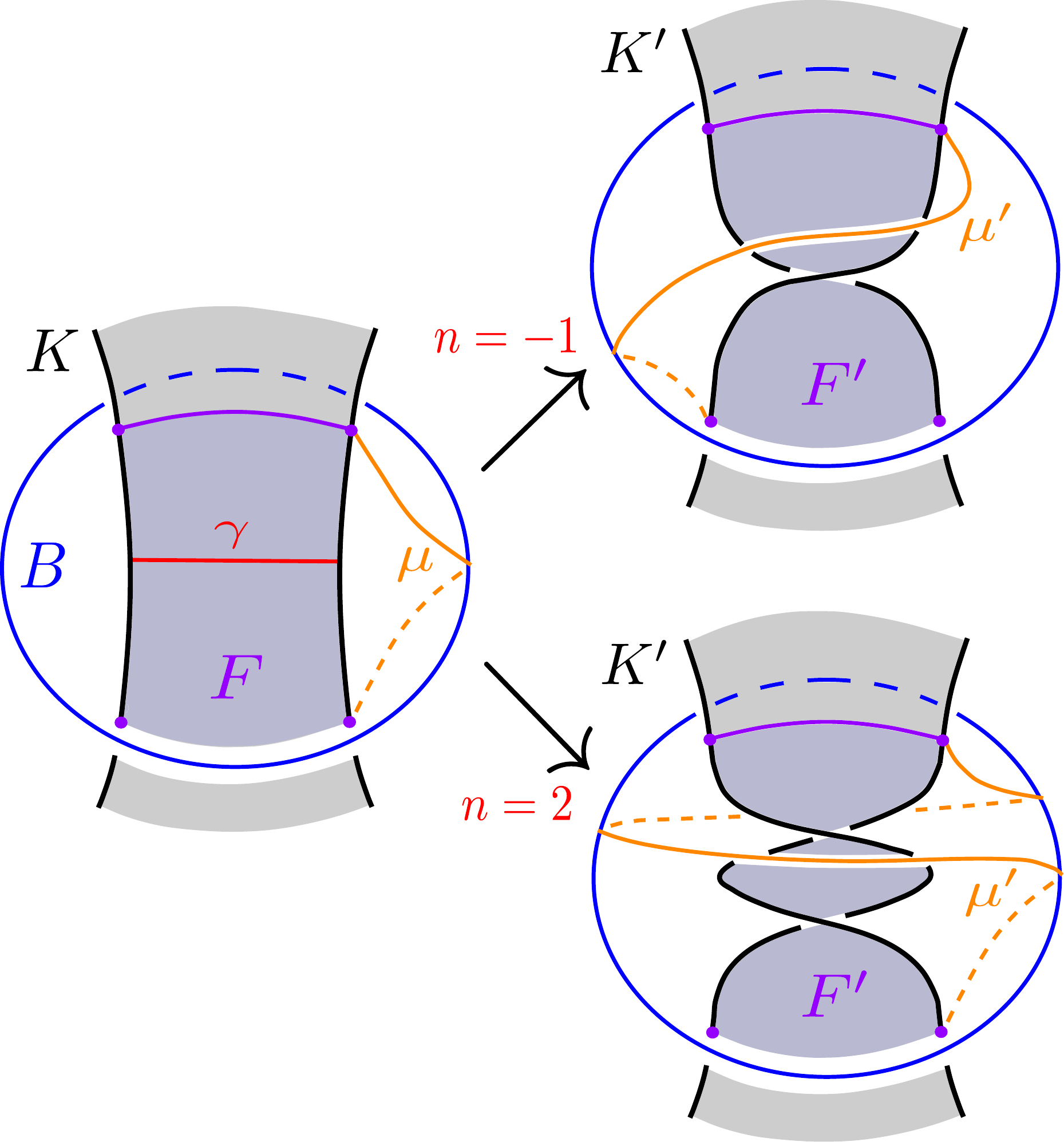}
	\caption{Instances of two-strand twists which act on a Seifert surface $F$ for $K$.}
	\label{montesinos1}
\end{figure}

By Definition \ref{maindef}, the pair $(M,K')$ is obtained from $(M,K)$ by regluing $(B,T)$ to $(\overline{M-B},\overline{K-T})$ via a map which can be viewed as an `$|n|/2$-fold Dehn twist' along a copy of the twisting curve $c$ embedded in $\partial B$. This is depicted in Figure \ref{montesinos1} for the cases $n=-1$ and $n=2$. That figure illustrates the effect of the regluing map on a \textit{meridian arc} for $K$ in $\partial B$. We define this to be an arc $\mu$ embedded in $\partial B$ for which there is a bigon $E$ embedded in $B$ such that: 
\begin{itemize}
\item
	$\partial E=\alpha\cup\mu$ for some component $\alpha$ of $B\cap K$. 
	
\item	
	$\mbox{Int}(E)$ is disjoint from both $\partial B$ and $K$.
\end{itemize}

This is precisely the type of arc on $\partial B$ which lifts to a meridian of $\widetilde{\gamma}$ in $\Sigma(K)$. The key observation is that the image $\mu'$ of $\mu$ under the regluing map is a meridian arc $\mu'$ for the new knot $K'$ which intersects $\mu$ in $|n|$ points in minimal position, relative to isotopies of $\partial B$ which fix each point of $K\cap\partial B$. This shows that the two-strand $n$-twist induces a Dehn surgery along $\widetilde{\gamma}$ whose slope has geometric intersection number $|n|$ with a meridian of $\widetilde{\gamma}$. Figure \ref{montesinos2} provides the picture on the torus for the case $n=2$, corresponding to the lower two-strand twist shown in Figure \ref{montesinos1}.

\begin{figure}[t]
	\centering
	\includegraphics[scale=0.4]{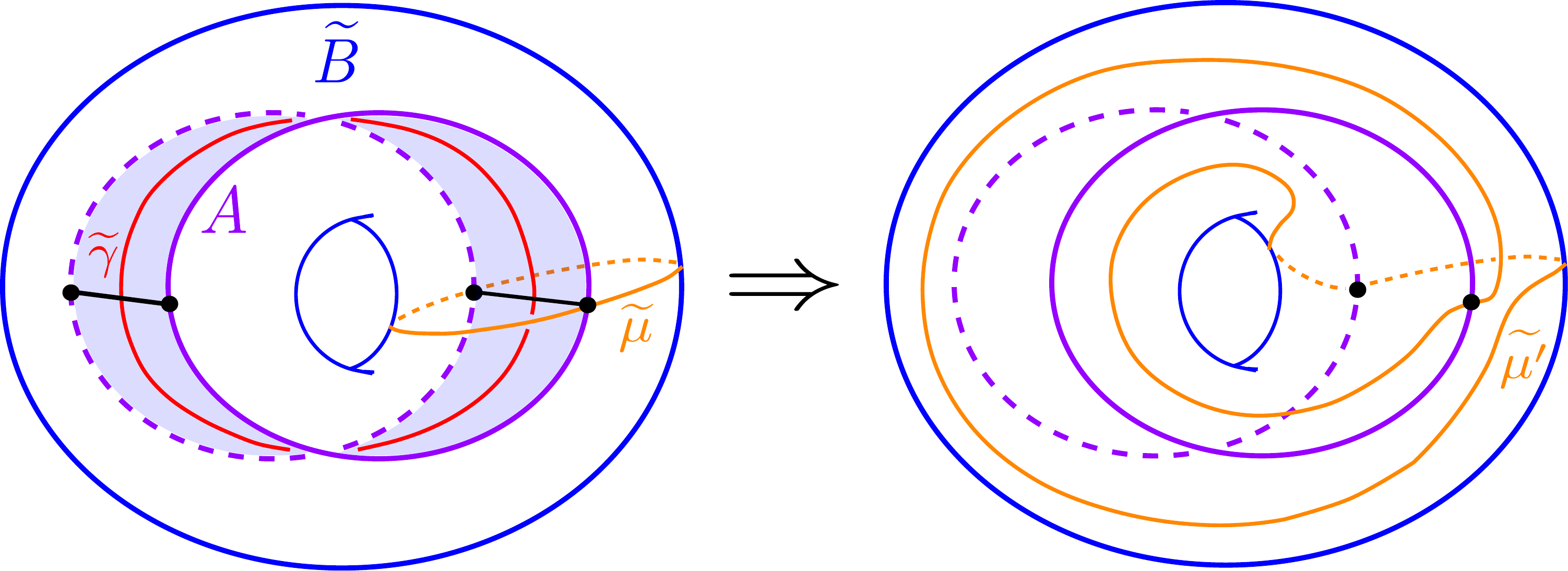}
	\caption{The Dehn surgery in $\Sigma(K)$ corresponding to the two-strand 2-twist shown in Figure \ref{montesinos1}.}
	\label{montesinos2}
\end{figure}

\begin{extseif}
If the twisting arc $\gamma$ is properly embedded in a Seifert surface $F$ for $K$, then we can refine this picture. As in Figure \ref{montesinos1}, we may assume that $F\cap B$ is a 4-gon which intersects $\partial B$ in two arcs. According to the labeling of that figure, the twisting which transforms the arc $\mu$ into the arc $\mu'$ occurs in the complement of $F\cap B$, so the corresponding two-strand twist naturally extends to $F$ to yield a new surface $F'$ bounded by $K'$.

In this situation, we say that $F'$ is obtained from $F$ by performing an \textit{$n$-twist} along $\gamma$. If $c$ is the corresponding twisting circle, as above, we will also refer to this operation as the $n$-twist on $F$ determined by $c$. Note that if $n$ is even or $\gamma$ separates $F$, then $F'$ will be orientable, and therefore a Seifert surface for $K'$.

In what follows, given an object $N$ embedded in $M$, let $\widetilde{N}$ denote the lift of $N$ to $\Sigma(K)$. Note that $\widetilde{\mu}$ is a meridian of $\widetilde{\gamma}$. The (closed) surface $\widetilde{F}$ intersects the solid torus $\widetilde{B}=\eta(\widetilde{\gamma})$ in an annulus $A$ such that each component of $\partial A$ is a longitude of $\widetilde{\gamma}$. Upon using the meridian $\widetilde{\mu}$ and this longitude to parametrize slopes on $\partial\widetilde{B}$, the lift $\widetilde{\mu'}$ of $\mu'$ to $\Sigma(K)$ visibly has slope $1/n$, where $n$ is as above. This is illustrated in Figure \ref{montesinos2} for the case $n=2$, which should be compared to the corresponding portion of Figure \ref{montesinos1}. We therefore have the following fact, which is a slight refinement of Lemma 2.1 of \cite{LM17}.
\end{extseif}

\begin{lem}\label{montlem}
Suppose that $K'$ is obtained from $K$ by performing a two-strand $n$-twist, and let $\widetilde{\gamma}$ be the lift of a corresponding twisting arc $\gamma$ for $K$ to $\Sigma(K)$. The manifold $\Sigma(K')$ is then obtained from $\Sigma(K)$ by Dehn surgery along $\widetilde{\gamma}$ with slope $\widetilde{\mu}'$ such that $\Delta(\widetilde{\mu},\widetilde{\mu}')=|n|$, where $\widetilde{\mu}$ is a meridian of $\widetilde{\gamma}$. Further, if $\gamma$ is properly embedded in a Seifert surface $F$ for $K$, then \[\widetilde{\mu'}=\frac{1}{n} \mbox{ with respect to the surface framing of }\widetilde{\gamma} \mbox{ determined by the lift } \widetilde{F} \mbox{ of F to }\Sigma(K).\]
\end{lem}

\noindent
Here, $\Delta(\alpha,\beta)$ denotes the minimal geometric intersection number between $\alpha$ and $\beta$.

As stated in the introduction, we say that a two-strand twist on $K$ is \textit{nugatory} if a corresponding twisting circle $c$ bounds a disk embedded in $M-K$. In this case, it follows readily that the lift of a corresponding twisting arc is unknotted in $\Sigma(K)$, meaning that it bounds an embedded disk. A key fact for us, as in the work of \cite{Tor99} and \cite{LM17}, is that a partial converse holds true.

\begin{propn}\label{liftprop}
Let $K$, $c$, and $\widetilde{\gamma}$ be as above. Suppose that $|n|\geq 2$, and let $K'$ be the knot obtained from $K$ by performing the two-strand $n$-twist determined by $c$. If $K'$ is isotopic to $K$ and $\widetilde{\gamma}$ is unknotted in $\Sigma(K)$, then the two-strand twist on $K$ determined by $c$ is nugatory. 
\end{propn}

This is more general than the corresponding Proposition 3.3 of \cite{LM17}. We now review the proof given there, and provide the additional argument needed to prove the above proposition. 

\begin{proof}
Let $N$ denote the exterior of $\widetilde{\gamma}$ in $\Sigma(K)$. In what follows, we freely identify simple closed curves on $\partial N$ with the corresponding elements of $H_1(\partial N)$. Since $\widetilde{\gamma}$ is an unknot, $N\cong (D^2\times S^1)\# \Sigma(K)$. Letting $\widetilde{\Gamma}:=D^2\times\{pt\}$, $\partial\widetilde{\Gamma}$ is the unique slope on $\partial N$ which bounds in $N$. Denote the covering involution on $\Sigma(K)$ by $\tau$. By the equivariant Dehn's lemma \cite{Ed86}, we may assume that either $\tau(\widetilde{\Gamma})\cap\widetilde{\Gamma}=\emptyset$ or $\tau(\widetilde{\Gamma})=\widetilde{\Gamma}$. As explained in \cite{LM17}, in the second case, $\widetilde{\Gamma}$ can in fact be perturbed to ensure that we are instead in the first case.

We therefore assume that $\tau(\widetilde{\Gamma})\cap\widetilde{\Gamma}=\emptyset$. Here, $\widetilde{\Gamma}$ descends to a properly embedded disk $\Gamma$ in $\overline{M_K-B}$, with boundary contained in $\partial B\cap M_K$, where $B$ is (as above) the 3-ball which lifts to $\eta(\widetilde{K})$. Our goal is to show that $\partial\Gamma$ is a copy of the twisting circle $c$, which will imply that the given two-strand twist is nugatory.

This is done in \cite{LM17} in the situation of a standard crossing change, which corresponds to the case $|n|=2$ in our setting. That argument works just as well in the case that $K$ is a null-homologous knot in a rational homology sphere. To treat the remaining cases, recall our previous notation: $\widetilde{\mu}$ is a meridian of $\widetilde{\gamma}$ and $\widetilde{\mu}'$ is the slope of the Dehn surgery along $\widetilde{\gamma}$ which yields $\Sigma(K')$, where $K'$ is the knot obtained from $K$ by the given two-strand twist. Let $\widetilde{\lambda}$ denote the longitude $\partial\widetilde{\Gamma}$ of $\widetilde{\gamma}$. We then have that $\Delta(\widetilde{\mu},\widetilde{\mu}')=|n|$ and $\Delta(\widetilde{\mu},\widetilde{\lambda})=1$. If $\Delta(\widetilde{\mu}',\widetilde{\lambda})$ were larger than one, then Dehn filling $N$ with slope $\widetilde{\mu}'$ would produce $L\sharp\Sigma(K)$ for some non-trivial lens space $L$. However, since the result of this filling is $\Sigma(K')$ and $K'$ is isotopic to $K$, $\Sigma(K')$ is orientation-preserving homeomorphic to $\Sigma(K)$. Thus, it must be that $\Delta(\widetilde{\mu}',\widetilde{\lambda})=1=\Delta(\widetilde{\mu},\widetilde{\lambda})$. It follows that, in the basis $(\widetilde{\mu},\widetilde{\lambda})$ of $H_1(\partial N)$, we have $\widetilde{\mu}'=\widetilde{\mu}\pm n\widetilde{\lambda}$. 

Now, suppose that $|n|>2$. Here, unlike in the cases $|n|=1$ and $|n|=2$, $\widetilde{\lambda}$ is the \textit{unique} slope on $\partial N$ intersecting both $\widetilde{\mu}$ and $\widetilde{\mu}'$ once geometrically. Note that any slope intersecting $\widetilde{\mu}$ exactly once is of the form $a\widetilde{\mu}\pm 1$ for some integer $a$, and by the above description of $\widetilde{\mu}'$, 
$\Delta(\widetilde{\mu}',a\widetilde{\mu}\pm 1)=|1\pm na|$. Since $|n|>2$, this can only equal 1 if $a=0$, in which case the slope is $\widetilde{\lambda}$. Since $\widetilde{\mu}$ and $\widetilde{\mu}'$ are lifts of the arcs $\mu$ and $\mu'$ in $\partial B$, and the twisting circle $c\subset\partial B$ intersects both $\mu$ and $\mu'$ once transversely, it follows that a lift of $c$ on $\partial N$ must have slope $\widetilde{\lambda}$. Consequently, the projection $\partial\Gamma$ of $\widetilde{\lambda}$ to $M$ is isotopic to $c$ in $\partial B-K$, so $c$ bounds a disk in $M_K$.
\end{proof}

The reader may be wondering about the need for the restriction $|n|\geq 2$ in this proposition. As we will show in Section 4, Proposition \ref{liftprop} is false when $n=\pm 1$. However, we will also see that a weaker version of the proposition does hold true in this case.

\section{Proof of Theorems \ref{mainthm1} and \ref{mainthm2}}

Through what follows, we assume that $K$ and $K'$ are knots in an oriented rational homology sphere $M$ such that: 
\begin{enumerate}
	\item
	$K$ and $K'$ are both fibered and have the same genus.
	
	\item
	$K'$ is obtained from $K$ by performing a two-strand $n$-twist determined by a twisting circle $c$.
\end{enumerate}

Note that we do \textit{not} immediately assume $K$ and $K'$ to be isotopic. To prove the two main theorems simultaneously, our first step is to observe that, in each case, the remaining hypotheses allow us to assume that the two-strand twist transforms a fiber surface for $K$ into one for $K'$, in the sense illustrated in Figure \ref{montesinos1}. This will pave the way for Theorem \ref{corethm}, which is the central result of this section. \\

We first treat the case relevant to Theorem \ref{mainthm2}. 

\begin{lem}\label{twistlem}
Let $\gamma$ be a twisting arc corresponding to the given two-strand twist. If $\gamma$ embeds in a fiber surface $F$ for $K$ as a separating arc, and if $c$ can be chosen to be disjoint from $F$, then the two-strand twist can be realized as an $n$-twist on $F$ along $\gamma$ which yields a fiber surface $F'$ for $K'$.	
\end{lem}
\begin{proof}
By our hypotheses, we may assume that $F$ intersects the 3-ball in which the two-strand twist occurs as on the left of Figure \ref{montesinos1} from Section 2.4. Then, as shown on the right of that figure, $F$ is transformed into a surface $F'$ bounded by $K'$ via an $n$-twist along $\gamma$. The fact that $\gamma$ separates $F$ means that $F'$ is orientable. Since $g(F')=g(F)=g(K)=g(K')$, it follows directly from Proposition \ref{unfib} that $F'$ is a fiber surface for $K'$.
\end{proof}

We now wish to reach the same conclusion in the situation of Theorem \ref{mainthm1}. Here, $K$ and $K'$ are related by a generalized crossing change ($n$ is even), but we have no explicit hypotheses on the twisting arc. This is addressed by repeating the logic of Section 5.2 of \cite{Kal12}. While we could simply observe that the proofs of Lemma 5.3 and Proposition 5.4 of that paper also work when $S^3$ is replaced with an arbitrary rational homology sphere, we provide the arguments for the sake of completeness. 

In addition, we wish to observe that the hypothesis of Kalfagianni's Proposition 5.4 can be weakened. Let $c$ denote the twisting circle which determines the given two-strand twist on $K$. In addition, let $L$ denote the link $K\cup c$ and $M_L=\overline{M-(\eta(K)\cup\eta(c))}$ denote the compact exterior of $L$, where $\eta(K)$ and $\eta(c)$ are disjoint tubular neighborhoods of $K$ and $c$ (respectively) in $M$. 

\begin{lem}
Let all notation be as above. If $M_L$ is reducible, then the given two-strand twist from $K$ to $K'$ is nugatory.
\end{lem}
\begin{proof}
Since $K$ is fibered, $M_K$ is irreducible. Consequently, if $S$ is an essential 2-sphere embedded in $M_L$, then $S$ must bound a 3-ball in $M$ which contains $c$ and is disjoint from $K$. Since $c$ is a crossing circle for $K$, it bounds a disk $D$ in $M$, and a standard innermost circle argument allows us to replace $D$ with a disk disjoint from $S$. This disk is then contained in $B$, and is therefore disjoint from $K$. By definition, this shows that the two-strand twist is nugatory.
\end{proof}

For all of what follows, we therefore assume that $M_L$ is irreducible. Since $K$ and $c$ are null-homologous knots in $M$ with linking number zero (after orienting them in any fashion), we know that $K$ is homologically trivial in the complement of $c$. Thus, we may choose a Seifert surface $F$ for $K$ disjoint from $c$ which is of minimum genus among all such surfaces. Let $F'$ denote the Seifert surface for $K'$ obtained from $F$ by performing the $n$-twist determined by $c$. 

\begin{propn}\label{twistprop}
Let all notation be as above. If the given two-strand $n$-twist is a generalized crossing change (so $n$ is even), then $F$ and $F'$ are fiber surfaces for $K$ and $K'$ (respectively).
\end{propn}
\begin{proof}
As discussed in Section 2.4, the fact that $n$ is even means that the $n$-twist is realized by performing Dehn surgery along $c$ with slope $-2/n$. Thus, the pairs $(M,F)$ and $(M,F')$ are obtained from $(M_L,F)$ by Dehn filling $M_L$ along $\partial\eta(c)$ in two different ways. By hypothesis on $F$, it is taut in the Thurston norm on $H_2(M_L,\partial\eta(L))$. Since $M_L$ is irreducible, Corollary 2.4 of \cite{Ga87} applies to this situation, and tells us that there is at most one slope $\alpha$ on $\partial\eta(c)$ such that $F$ fails to remain taut in the manifold obtained by Dehn filling $M_L$ along $\partial\eta(c)$ via $\alpha$. It follows that $K$ and $K'$ cannot both bound Seifert surfaces of genus less than $g(F)=g(F')$. Since $g(K)=g(K')$, it follows that $F$ and $F'$ must be minimum genus Seifert surfaces for $K$ and $K'$ (respectively). By Proposition \ref{unfib}, this means that $F$ and $F'$ are fiber surfaces.
\end{proof}

\begin{adhyp}
Justified by Lemma \ref{twistlem} and Proposition \ref{twistprop}, we now assume: 
\begin{enumerate}
\item
	$F$ is a fiber surface for $K$ disjoint from the twisting circle $c$.

\item
	The result $F'$ of performing the $n$-twist on $F$ determined by $c$ is a fiber surface for $K'$.
\end{enumerate}

Let $D$ be the twisting disk for $K$ bounded by $c$. Since $F$ is incompressible and $M_K$ is irreducible, we may apply an innermost circle argument to isotope $F$ (rel boundary) in the complement of $c$ so that $F\cap D$ is reduced to a single arc $\gamma$ which is properly embedded in $F$. By definition, $\gamma$ is a twisting arc corresponding to the given two-strand twist. In the setting of Theorem \ref{mainthm2}, we further assume that $\gamma$ is the specified separating arc in $F$.
\end{adhyp}

To prepare for the central argument, we need to fix some new notation.

\begin{setup}
Recall that $\Sigma(K)$ denotes the cyclic double cover of $M$ branched over $K$. The lift of a manifold $N$ embedded in $M$ to $\Sigma(K)$ will be denoted by $\widetilde{N}$, and similarly for lifts of objects to $\Sigma(K')$. Let $\eta(\widetilde{K})$ be the tubular neighborhood of $\widetilde{K}$ in $\Sigma(K)$ that projects onto the chosen tubular neighborhood $\eta(K)$ of $K$, for which $M_K=\overline{M-\eta(K)}$. Through what follows, let $\widetilde{\Sigma}(K)=\overline{\Sigma(K)-\eta(\widetilde{K})}$.
\end{setup}

Note that since $\partial F=K$, $\widetilde{F}$ is a closed, separating surface in $\Sigma(K)$, and $\eta(\widetilde{K})\cap\widetilde{F}$ is an annulus properly embedded in $\eta(\widetilde{K})$. Since $\gamma$ is properly embedded in $F$, $\widetilde{\gamma}$ is a simple closed curve in $\widetilde{F}$.

A key fact is that the fibration of $M_K$ by parallel copies of $F$ induces a fibration of $\widetilde{\Sigma}(K)$ for which $\widetilde{F}\cap\widetilde{\Sigma}(K)$ consists of two disjoint fibers. By our choice of $\eta(\widetilde{K})$, each fiber in this fibration projects homeomorphically onto a copy of $F$ in $M_K$, so we may naturally identify these fibers with $F$ itself. Further, if $\phi$ is a monodromy for $K$, then we may identify $\widetilde{\Sigma}(K)$ with $F\times[-1,1]/\phi^2$ so that $\widetilde{F}\cap \widetilde{\Sigma}(K)$ corresponds to the union of $F\times\{-1\}$ with $F\times\{0\}$. This all arises from the fact that $\widetilde{\Sigma}(K)$, the 2-fold cyclic cover of $M_K$, can be constructed by taking two copies of $M_K$ cut along $F$ and gluing them together end-to-end.

Recall that $K'$ is also fibered with fiber surface $F'$, which is obtained from $F$ by performing an $n$-twist along $\gamma$. Through this identification of $F'$ with $F$, we may view the monodromy for $K'$ as an orientation-preserving homeomorphism $\psi:F\rightarrow F$. By the above reasoning, letting $\widetilde{K'}$ denote the lift of $K'$ to $\Sigma(K')$, we obtain a natural identification of $\widetilde{\Sigma}(K')$ with $F\times[-1,1]/\psi^2$. The following result is the heart of our argument.

\begin{thm}\label{corethm}
Let $K$ and $K'$ be two fibered knots in $M$ of the same genus related by a two-strand $n$-twist, as above, with corresponding monodromies $\phi$ and $\psi$. Suppose that the two-strand twist extends to an $n$-twist on a fiber surface $F$ for $K$ along the arc $\gamma$, so that the resulting surface $F'$ is a fiber surface for $K'$. Then, if $|n|>1$, there is a simple closed curve $\alpha$ in $F$ such that: 
\begin{description}
\item{(a) }
	When $\alpha$ is viewed as lying in $F\times\{-1\}$ within $\widetilde{\Sigma}(K)=F\times[-1,1]/\phi^2$, the curve $\widetilde{\gamma}$ is isotopic to $\alpha$ in $\Sigma(K)$.

\item{(b) }
	$[T^n_{\alpha}][\phi^2]=[h][\psi^2][h^{-1}]$ in $\mbox{MCG}(F)$ for some orientation-preserving homeomorphism \\ $h:F\rightarrow F$ which restricts to the identity on $\partial F$.
\end{description}
\end{thm}

To prove this theorem, we consider the Heegaard splitting $\Sigma(K)=V_0\cup_SV_1$ of $\Sigma(K)$ induced by the fibration of $\widetilde{\Sigma}(K)$ described above, as defined in Section 2.1.2. Our conventions have been chosen so that the Heegaard surface $S$ is precisely the lift $\widetilde{F}$ of the distinguished fiber surface $F$ for $K$. We likewise consider the Heegaard splitting $\Sigma(K')=V'_0\cup_{\widetilde{F'}}V'_1$ induced by the fibration of $\widetilde{\Sigma}(K')$. 

By Lemma \ref{montlem}, $\Sigma(K')$ is obtained from $\Sigma(K)$ by Dehn surgery along $\widetilde{\gamma}$ with slope $1/n$, when parametrized with respect to the framing of $\widetilde{\gamma}$ determined by $\widetilde{F}$. Since the given two-strand twist acts as an $n$-twist transforming $F$ into $F'$, this surgery transforms $\widetilde{F}$ into a copy of the Heegaard surface $\widetilde{F'}$ for $\Sigma(K')$. A key observation is that this transformation may be reinterpreted in terms of gluing maps for the two Heegaard splittings. This correspondence is utilized by Kalfagianni in the context of Heegaard splittings of $S^3$: see Lemma 5.5 of \cite{Kal12}. Since our conventions and terminology differ from hers, we provide a complete description.

For a Heegaard splitting $V\cup_SW$ of an oriented 3-manifold, where $V$ and $W$ are viewed as copies of a single handlebody with boundary $S$, we view a gluing map for the splitting as an orientation-reversing homeomorphism $h:S\rightarrow S$ such that $M$ is formed from $V$ and $W$ by identifying $x\in\partial V$ with $h(x)\in\partial W$. Here, we take $S$ to be oriented so that its positive normal vector points into $W$ everywhere. If $g:S\rightarrow S$ is an orientation-preserving homeomorphism, we use $V\cup_{g(S)}W$ to denote the Heegaard spitting (likely of a different 3-manifold) whose gluing map is $g\circ h$.

\begin{lem}\label{gluelem}
Let all notation be as above. Then $V_0\cup_{T^n_{\widetilde{\gamma}}(\widetilde{F})}V_1$ is a Heegaard splitting of $\Sigma(K')$ which is equivalent to $V'_0\cup_{\widetilde{F'}}V'_1$. Further, these Heegaard splittings are related by an orientation-preserving homeomorphism which takes $T^n_{\widetilde{\gamma}}(\widetilde{K})\subset\partial V_1$ to $\widetilde{K'}\subset\partial V'_1$ and is fiber-preserving on the exteriors of these knots.
\end{lem}

\begin{proof}
Up to orientation-preserving homeomorphism, we can construct $V_0\cup_{T^n_{\widetilde{\gamma}}(\widetilde{F})}V_1$ from $\Sigma(K)=V_0\cup_{\widetilde{F}}V_1$ as follows. Choose a small collar $N(\widetilde{F})\cong\widetilde{F}\times[0,1]$ on $\widetilde{F}$ in $V_1$, so that $\widetilde{F}$ is identified with $\widetilde{F}\times\{0\}$. Let $h:\widetilde{F}\rightarrow\widetilde{F}$ be the gluing map for the splitting $V_0\cup_{\widetilde{F}} V_1$. Instead of forming $V_0\cup_{T^n_{\widetilde{\gamma}}(\widetilde{F})}V_1$ by gluing $V_0$ to $V_1$ via $T^{\pm n}_{\widetilde{\gamma}}h$, we can first glue $V_0$ to $N(\widetilde{F})$ along $\widetilde{F}\times\{0\}$ via $h$, and then glue $N(\widetilde{F})$ to $\overline{V_1-N(\widetilde{F})}$ via $T^n_{\widetilde{\gamma}'}:\widetilde{F}\times\{1\}\rightarrow\widetilde{F}\times\{1\}$, where $\widetilde{\gamma}'=\widetilde{\gamma}\times\{1\}$.

This process is equivalent to taking a small annular neighborhood $A$ of $\widetilde{\gamma}$ in $\widetilde{F}$, removing $A\times[0,1]\subset N(\widetilde{F})$ from $\Sigma(K)$, and gluing it back in by the map indicated in Figure \ref{twist_surg}. This is precisely the operation of performing Dehn surgery along $\widetilde{\gamma}\times\{1/2\}$ with slope $1/n$, as parametrized with respect to the surface framing. Observe that this operation produces a new Heegaard splitting $[V_0\cup N(\widetilde{F})]\cup_{\widetilde{F}\times\{1\}}[\overline{V_1-N(\widetilde{F})}]$ of $\Sigma(K')$. This is evidently related to the original splitting $V_0\cup_{T^n_{\widetilde{\gamma}}(\widetilde{F})}V_1$ by an orientation-preserving homeomorphism which identifies $T^n_{\widetilde{\gamma}}(\widetilde{K})$ with $T^n_{\widetilde{\gamma}'}(\widetilde{K}\times\{1\})$, as well as the fibrations of the exteriors of these knots. 

This construction therefore recovers $\Sigma(K')$ by Lemma \ref{montlem}, and turns any copy of $\widetilde{F}$ into one of $\widetilde{F'}$. Further, the natural copy of $\widetilde{K}$ in $\widetilde{F}\times\{t\}$ intersecting $A\times\{t\}$ in two properly embedded arcs becomes isotopic to a copy of $\widetilde{K'}$, within $\widetilde{F}\times\{t\}$, after performing the surgery described above. These identifications give rise to an orientation-preserving homeomorphism $\Sigma(K')\rightarrow\Sigma(K')$ which takes $\widetilde{F}\times\{1\}$ to $\widetilde{F'}$, $V_0\cup_{\widetilde{F}}N(\widetilde{F})$ to $V'_0$, and $T^n_{\widetilde{\gamma'}}(\widetilde{K})$ to $\widetilde{K'}$. This homeomorphism necessarily respects the decompositions of $\widetilde{F}\times\{1\}$ and $\widetilde{F'}$ into two copies of the fiber surface for these knots, so by Lemma 3.5 of \cite{Wa68}, it may be isotoped so that it becomes an entirely fiber-preserving homeomorphism from the exterior of $T^n_{\widetilde{\gamma'}}(\widetilde{K})$ to $\widetilde{\Sigma}(K')$. Combined with the last observation of the previous paragraph, this completes the proof. 
\end{proof}

\begin{figure}[t]
	\centering
	\includegraphics[scale=0.6]{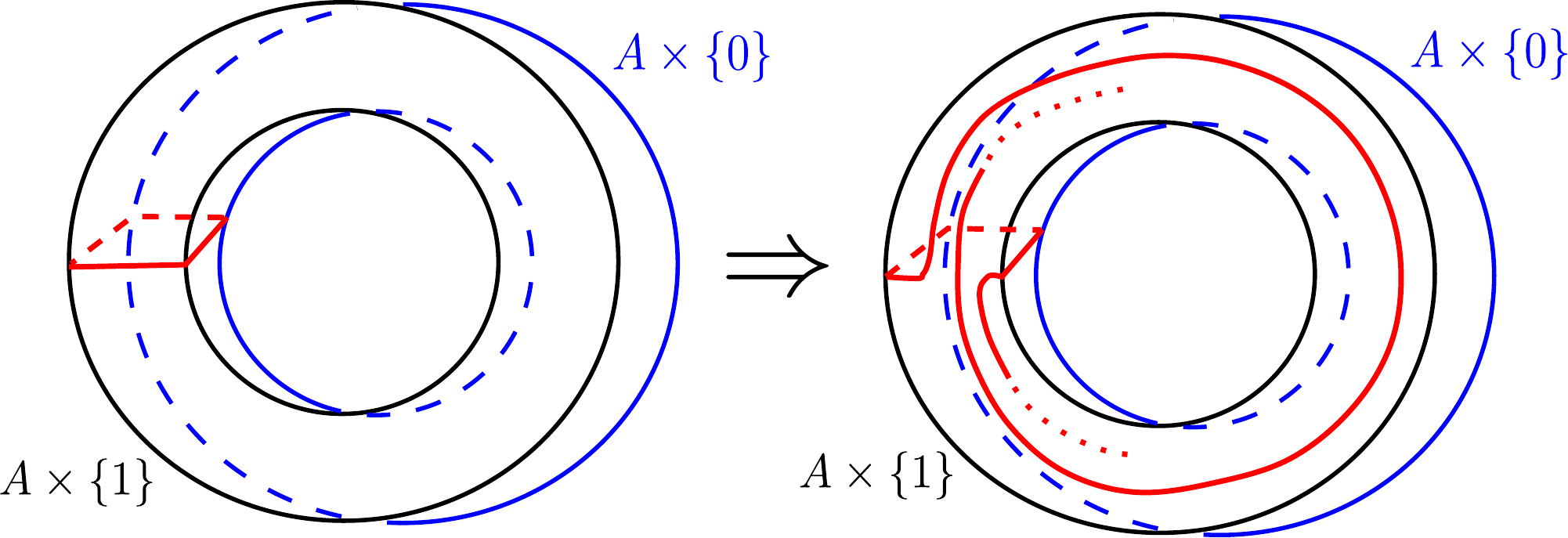}
	\caption{Realizing a change in a gluing map by an $n$-fold left Dehn twist via Dehn surgery. On the right side, the surgery curve intersects the annulus $A\times\{1\}$ in an arc which wraps around the annulus $n$ times.}
	\label{twist_surg}
\end{figure}

\begin{proof}[of Theorem~{\rm\ref{corethm}}]
Within $\Sigma(K)=V_0\cup_{\widetilde{F}}V_1$, we assume that $V_0$ consists of \\ $F\times[-1,0]\subset\widetilde{\Sigma}(K)$ together with part of the tubular neighborhood $\eta(\widetilde{K})$, as in Section 2.1.2. Consider the subsurface $\widehat{F}$ of $F\times\{-1/2\}$ formed by shrinking it slightly to be disjoint from the lift $\widetilde{B}$ of the 3-ball in which the two-strand twist occurs. There is then a natural bicollar $N(\widehat{F})\cong\widehat{F}\times[-1,0]$ on $\widehat{F}$ in $V_0$ disjoint from $\widetilde{B}$. Since $\widehat{K}=\partial\widehat{F}$ is a fibered knot in $\Sigma(K)$ (isotopic to $\widetilde{K}$) with fiber surface $\widehat{F}$, the surface exterior $W=\overline{\Sigma(K)-N(\widehat{F})}$ is homeomorphic to $\widehat{F}\times[0,1]$, so that $\widehat{F}\times\{0,1\}$ is identified with the two natural copies of $\widehat{F}$ in $\partial N(\widehat{F})$.

At the same time, under the identification $\Sigma(K')\cong V_0\cup_{T^n_{\widetilde{\gamma}}(\widetilde{F})}V_1$ from Lemma \ref{gluelem}, $\widehat{F}$ gives rise to a surface $\widehat{F}'$ in $V_0\subset\Sigma(K')$ whose exterior $W'=\overline{\Sigma(K')-N(\widehat{F}')}$ is obtained from $W$ by the  $1/n$-sloped Dehn surgery along $\widetilde{\gamma}$ corresponding to the Dehn twist added to the gluing map. Further, by our construction and the second statement of Lemma \ref{gluelem}, $\widehat{F}'$ is isotopic in $V_0$ to a fiber surface for $K'$.

This implies that $W'$ is homeomorphic to $\widehat{F}'\times[0,1]$, so that the product structure agrees with that of $W$ on the boundary under the natural identification between $\widehat{F}$ and $\widehat{F}'$. By Theorem \ref{nithm}, it follows that the surgery curve $\widetilde{\gamma}$ is a 0 or 1-crossing knot in $W$. The surgery slope is non-integral, due to the fact that $|n|>1$, so Theorem \ref{nithm} further implies that $\widetilde{\gamma}$ must be a 0-crossing knot. This means that $\widetilde{\gamma}$ is isotopic in $W'$ to a simple closed curve $\alpha$ in $\widehat{F}\times\{-1\}\subset\partial N(\widehat{F})$. The last piece of information arising from Ni's theorem is that the corresponding surgery slope along $\alpha$ is $1/k$ for some integer $k\neq 0$, with respect to the framing of $\alpha$ determined by $F\times\{-1\}$. Since $|n|>1$ and the original surgery slope on $\widetilde{\gamma}$ is $1/n$ with respect to the surface framing given by $\widetilde{F}$, it follows that $k=n$.

Since $\widehat{F}'\times\{-1\}$ is a slightly deformed copy of a fiber surface for $\widetilde{K'}$, we may instead view $\alpha$ as contained in $F\times\{-1\}$ within $M_{K'}\cong F\times[-1,1]/\psi^2$. Statement (b) of Theorem \ref{corethm} now arises from a correspondence between modifications of monodromies by Dehn twists and certain Dehn surgeries, along the exact same lines of the correspondence contained in Lemma \ref{gluelem}. Let $N=F\times[-1,1]/h$ for some orientation-preserving homeomorphism $h:F\rightarrow F$. If there is a curve $\alpha\subset F$ such that $N'$ is obtained from $N$ by Dehn surgery along $\alpha\times\{-1\}$ of slope $1/n$ with respect to the surface framing, then $N'=F\times[-1,1]/T^n_{\alpha}h$. Since this is essentially contained in Theorem 1.4 of \cite{Ni11} and its proof is just like that of Lemma \ref{gluelem}, we do not provide further details here.

Together with our previous observations, this means that $\widetilde{\Sigma}(K')\cong F\times[-1,1]/T^n_{\alpha}\phi^2$, since $\phi^2$ is the monodromy for $\widetilde{K}$. Thus, both $T^n_{\alpha}\phi^2$ and $\psi^2$ are monodromies for the same fibration of $\widetilde{\Sigma}(K')$. As remarked following Proposition \ref{unmon}, the equivalence of parts (a) and (b) of that proposition holds regardless of the topology of the ambient closed 3-manifold (here, $\Sigma(K')$), so it follows that $T^n_{\alpha}\phi^2$ and $\psi^2$ represent conjugate elements of $\mbox{MCG}(F)$. This establishes statement (b) of Theorem \ref{corethm}, and therefore completes the proof.
\end{proof}

\begin{proof}[of Theorem~{\rm\ref{mainthm1}} and Theorem~{\rm\ref{mainthm2}}]
So far, the only hypotheses we have taken on $K$ and $K'$ are that they are fibered knots in a rational homology sphere $M$ of the same genus, related either by a generalized crossing change or a two-strand $n$-twist with $|n|>1$ satisfying the twisting arc hypothesis of Theorem \ref{mainthm2}. As discussed following Proposition \ref{twistprop}, that proposition together with Lemma \ref{twistlem} enable us to assume that this two-strand twist acts as an $n$-twist on a fiber surface $F$ for $K$ along a properly embedded arc $\gamma$, so that the resulting surface $F'$ is a fiber surface for $K'$. Let $\phi:F\rightarrow F$ denote a monodromy for $K$ and $\psi$ denote one for $K'$, which (as above) is also viewed as a  homeomorphism $F\rightarrow F$.

We finally assume that $K$ and $K'$ are isotopic knots in $M$. By Proposition \ref{unmon}, this implies that there is an orientation-preserving homeomorphism $g:F\rightarrow F$, restricting to the identity on $\partial F$, such that $[\psi]=[g][\phi][g^{-1}]$ in $\mbox{MCG}(F)$. It follows that $[\psi^2]=[g][\phi^2][g^{-1}]$. Since $|n|>1$, Theorem \ref{corethm} then tells us that $[T^n_{\alpha}][\phi^2]=[hg][\phi^2][g^{-1}h^{-1}]$ for some simple closed curve $\alpha$ in $F$ and homeomorphism $h:F\rightarrow F$ satisfying the same properties as $g$. Further, $\alpha$ is isotopic to the lift $\widetilde{\gamma}$ of the twisting arc $\gamma$ to $\Sigma(K)$.

This means that $[T^n_{\alpha}]=[hg][\phi^2][(hg)^{-1}][\phi^{-2}]$, so that $T^n_{\alpha}$ represents a commutator in $\mbox{MCG}(F)$. By Corollary \ref{kotcor}, it follows directly that $\alpha$ is homotopically trivial in $F$, and therefore that the isotopic curve $\widetilde{\gamma}$ is unknotted in $\Sigma(K)$. Since $|n|>1$, we may now apply Proposition \ref{liftprop} to conclude that the given two-strand $n$-twist on $K$ must be nugatory.

In particular, this shows that a generalized crossing change from a fibered knot in $M$ to an isotopic knot must be nugatory, which is precisely the statement of Theorem \ref{mainthm1}. The remaining content of what we just proved is as follows: if two isotopic fibered knots in $M$ are related by a two-strand $n$-twist with $|n|>1$, the twisting circle is disjoint from a fiber surface $F$ for one of them, and a corresponding twisting arc lies in and separates $F$, then this twist must be nugatory.

In other words, if we are under the hypotheses of Theorem \ref{mainthm2} and a two-strand $n$-twist on $K$ determined by $c$ is cosmetic, then $n=\pm 1$. The fact that the two-strand twist can be cosmetic for at most one of these values of $n$ is now a consequence of Theorem \ref{mainthm1}. Note that for any knot, the two-strand $-1$ and $1$-twists determined by the same twisting circle are related by a standard crossing change. Thus, if both such twists on $K$ determined by $c$ are isotopic to $K$, then both must be nugatory by Theorem \ref{mainthm1}. We therefore conclude that the two-strand twist on $K$ determined by $c$ is cosmetic for at most one integer, which must be among -1 and 1. This is precisely the claim of Theorem \ref{mainthm2}.
\end{proof}

\begin{rmk}
It should be noted that Buck, Ishihara, Rathbun, and Shimokawa have previously applied Theorem \ref{nithm} to give a complete classification of generalized crossing changes on fibered knots in $S^3$ which preserve genus \cite{BIRS16}. A curious reader is referred to Theorem 5 and Corollary 5 of their paper. Theorem \ref{corethm} certainly does not do this, but provides the information needed for our purposes.

In an earlier stage of this work, the author realized that the results of \cite{BIRS16} can be combined with Corollary \ref{kotcor} to quickly prove Theorem \ref{mainthm1} in the case of generalized crossing changes of order $\geq 2$ within $M=S^3$. It is unclear if their work may lead to an alternative proof of Theorem \ref{mainthm1} in full. It cannot, however, be used to say anything about two-strand $n$-twists for odd $n$. 
\end{rmk}

\section{Two-strand twists of odd order}

Define the \textit{order} of a two-strand $n$-twist on a knot $K$ to be $|n|$. The purpose of this final section is two-fold. For one, we show why results on generalized crossing changes should not generally be expected to extend to two-strand twists of odd orders. In Section 4.1, we exhibit two simple but noteworthy examples of cosmetic two-strand twists of orders 1 and 5. They show that the direct extension of Conjecture \ref{ccc} to two-strand $n$-twists is false when $n$ is odd, both when $n=\pm 1$ and for fibered knots. Further, these examples will reveal that there is no simple way to weaken the hypotheses of Theorem \ref{mainthm2}, and that the hypothesis $|n|\geq 2$ cannot be removed from Proposition \ref{liftprop}. 

The second purpose of this section is to provide a closer look at the interesting case of two-strand twists of order 1. These are precisely the types of two-strand twists that can be viewed as \textit{non-coherent band surgeries}, as discussed in \cite{IJ18} and \cite{MV18}. Our first example in Section 4.1 is of a certain two-strand $-1$-twist on the unknot in $S^3$. As we will discuss there, this is an instance of non-trivial purely cosmetic band surgery, in the sense of `non-trivial' that appears to be intended in \cite{IJ18}. However, in Section 4.2, we make the case that this example should be viewed as `trivial' in a broader sense, which we refer to as \textit{weakly nugatory}.

After making this notion precise in Definition \ref{weaknug}, we show that Proposition \ref{liftprop} becomes true in the case $n=\pm 1$ if one replaces `nugatory' with `weakly nugatory' in its conclusion. This provides a new avenue for using the Montesinos trick to study purely cosmetic band surgeries. As an immediate application, we prove Theorem \ref{mainthm3}. Following the proof, we discuss the difficulty in using Proposition \ref{liftprop} to extend the reasoning of Section 3 to improve the conclusion of Theorem \ref{mainthm2} when $n=\pm 1$.

\subsection{Examples}

Our two examples of cosmetic two-strand twists of odd orders are related to each other, and easy to describe. In the first example, we take time to discuss what it shows in regard to a question about band surgeries on knots posed in \cite{IJ18}. We refer the reader to that paper and Section 3 of \cite{MV18} for standard definitions surrounding band surgery.

\begin{ex}
	Consider the two-strand $-1$-twist on the unknot $K$ shown in Figure \ref{cosmetic_unknot}. The resulting knot is evidently another unknot. One simple way to see that this twist is not nugatory is to observe that performing the two-strand 2-twist on $K$ determined by the same twisting circle $c$ produces the figure-eight knot $J$ shown on the left side of Figure \ref{cosmetic_fig8}. Thus, $c$ cannot bound a disk in the complement of $K$, as $J$ would be isotopic to $K$ otherwise. 
	
	\begin{figure}[t]
		\centering
		\includegraphics[scale=0.5]{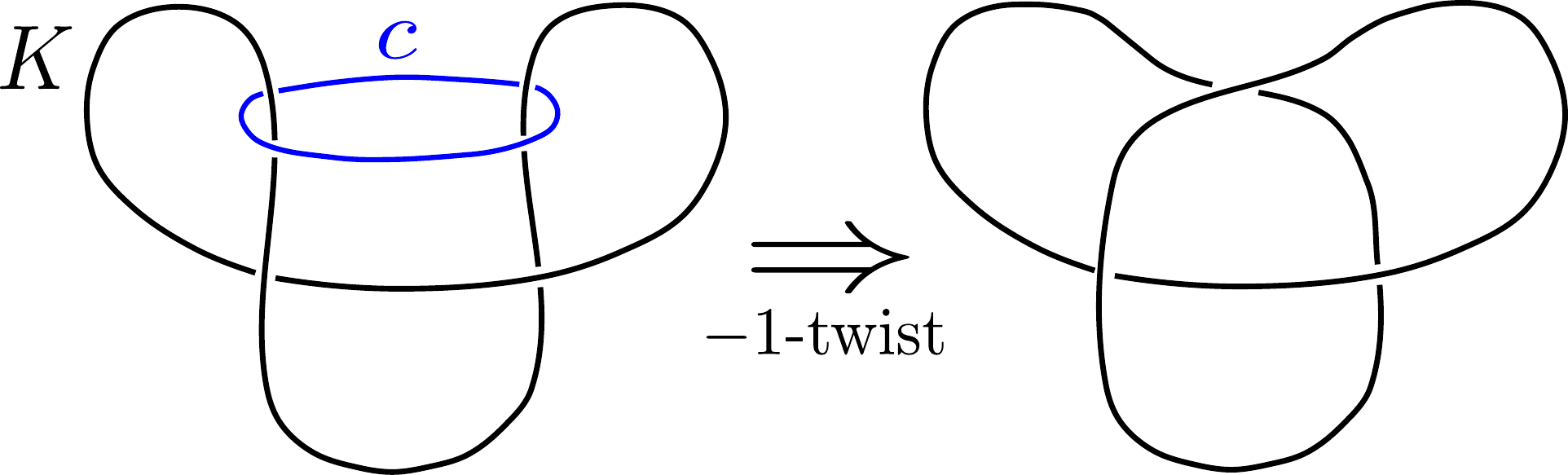}
		\caption{A cosmetic two-strand $-1$-twist on the unknot in $S^3$.}
		\label{cosmetic_unknot}
	\end{figure}

	When interpreted in the language of band surgery, this provides an example of what appears to be a previously unobserved phenomenon. As mentioned before, two-strand twists of order one are instances of \textit{non-coherent} band surgeries, which transform a knot into another knot (as opposed to a 2-component link). In the introduction of \cite{IJ18}, it is noted that there are no previously known examples of non-trivial \textit{purely cosmetic} band surgeries, which are non-coherent band surgeries taking a knot to an isotopic knot. The meaning of `trivial' provided there is `when the band is half-twisted and parallel to the link.'
	
	In our language of two-strand $\pm 1$-twists, the best interpretation of this is illustrated in Figure \ref{trivial_band}: there should be a bigon $E$ embedded in $M$ with interior disjoint from $K$ and $\partial E=\alpha\cup\gamma$, where $\alpha$ is an arc in $K$ and $\gamma$ is a twisting arc corresponding to the given $\pm 1$-twist. In this case, it follows right away that the twist is nugatory. Of course, the converse is not true: by starting with such a trivial banding on a non-trivial knot and adding local knotting to the distinguished arc $\alpha$, we obtain nugatory two-strand $\pm 1$-twists on composite knots that are not trivial in this sense.
	
	\begin{figure}[b]
		\centering
		\includegraphics[scale=0.65]{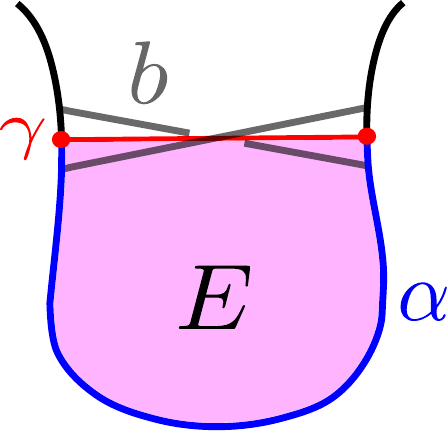}
		\caption{An instance of trivial non-coherent band surgery.}
		\label{trivial_band}
	\end{figure}
	
	As we imagine that the authors of \cite{IJ18} would not regard such an example as truly non-trivial, we define a non-coherent band surgery on a knot $K$ to be \textit{trivial} if it can be realized as a nugatory two-strand twist of order one. Thus, the example of Figure \ref{cosmetic_unknot} is a non-trivial purely cosmetic band surgery on the unknot, under this definition. By connect-summing the unknots on each side of that figure with any other knot in a consistent way, it follows that every knot admits a non-trivial purely cosmetic band surgery. However, in Section 4.2, we will see that these examples are still quite close to being trivial.
\end{ex}

\begin{rmk}\label{surgtwist}
A non-coherent band surgery on a knot $K$ in $M$ can always be realized \textit{uniquely} as a two-strand twist of order 1, in the sense that the choice of band $b$ determines a corresponding twisting circle up to isotopy in $M_K$. Namely, parametrize $b$ as $I\times I$ with $I=[0,1]$ so that $K\cap b=I\times\{0,1\}$. Let $K'$ denote the knot obtained from $K$ by performing surgery along $b$. The isotopy class of the twisting circle $c$ is then specified by the requirements that: 
\begin{enumerate}
	\item
	$\{1/2\}\times I$ is a corresponding twisting arc. 
	
	\item
	Some twisting disk bounded by $c$ intersects both $K$ and $K'$ in exactly two points of opposite sign, upon orienting each knot somehow.
\end{enumerate}
\end{rmk}

\begin{ex} 
	For our second example of a cosmetic two-strand twist of odd order, note that performing the $-3$-twist determined by $c$ on the unknot $K$ from the previous example also produces the figure-eight knot $J'$ appearing on the right side of Figure \ref{cosmetic_fig8}. Viewed another way, $J'$ is obtained from $J$ by performing a two-strand $-5$-twist along $c$. Note that $c$ cannot bound a disk in the complement of $J$: otherwise, it would also bound a disk in the complement of the unknot $K$ from Figure \ref{cosmetic_unknot}.
	
	Since the figure-eight knot is achiral, this $-5$-twist is therefore cosmetic. At the same time, the figure-eight is fibered with genus one fiber surface depicted on the left of Figure \ref{cosmetic_fig8}, and the given twist can be realized as a $-5$-twist performed along the arc shown in red. Thus, the separating hypothesis on the twisting arc in Theorem \ref{mainthm2} cannot be removed.
\end{ex}

\begin{figure}[t]
	\centering
	\includegraphics[scale=0.5]{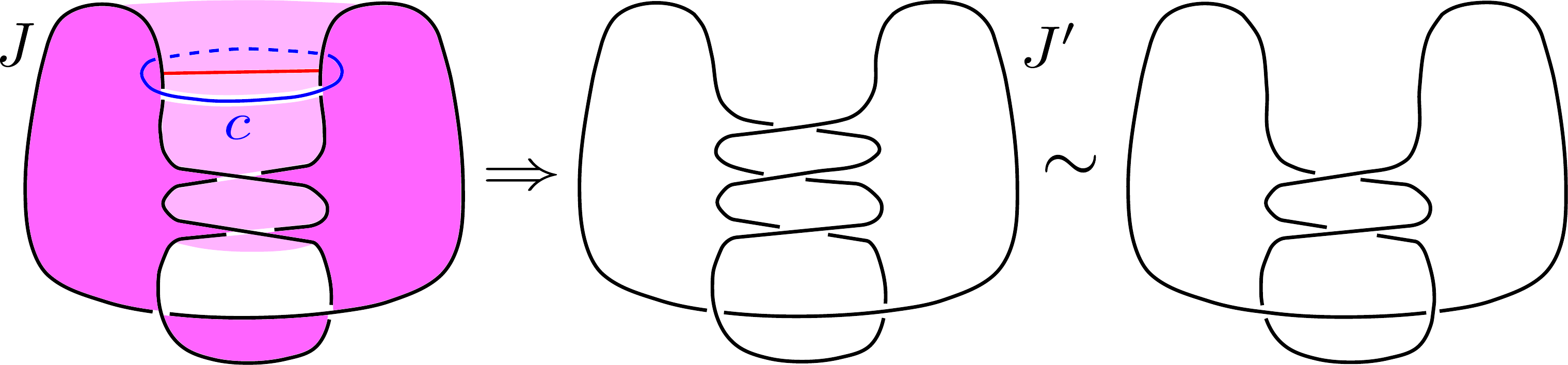}
	\caption{A cosmetic two-strand $-5$-twist on the figure-eight knot, performed along a non-separating arc in its fiber surface.}
	\label{cosmetic_fig8}
\end{figure}

\subsection{The case of non-coherent band surgery}

We begin this subsection by taking a closer look at the two-strand $-1$-twist on the unknot exhibited above. We continue to denote the initial version of the unknot, as on the left of Figure \ref{cosmetic_unknot}, by $K$. By the Montesinos trick, as summarized in Lemma \ref{montlem}, a corresponding twisting arc $\gamma$ lifts to a knot $\widetilde{\gamma}$ in $\Sigma(K)\cong S^3$ which admits a non-trivial Dehn surgery producing $S^3$. By the Gordon-Luecke theorem \cite{GL89}, it follows that $\widetilde{\gamma}$ must be an unknot. (Alternatively, one may be able to see this by directly constructing a diagram of $\widetilde{\gamma}$ in $\Sigma(K)$.) Since the corresponding $-1$-twist is not nugatory, this shows that the conclusion of Proposition \ref{liftprop} is false when $|n|=1$.

\begin{figure}[t]
	\centering
	\includegraphics[scale=0.35]{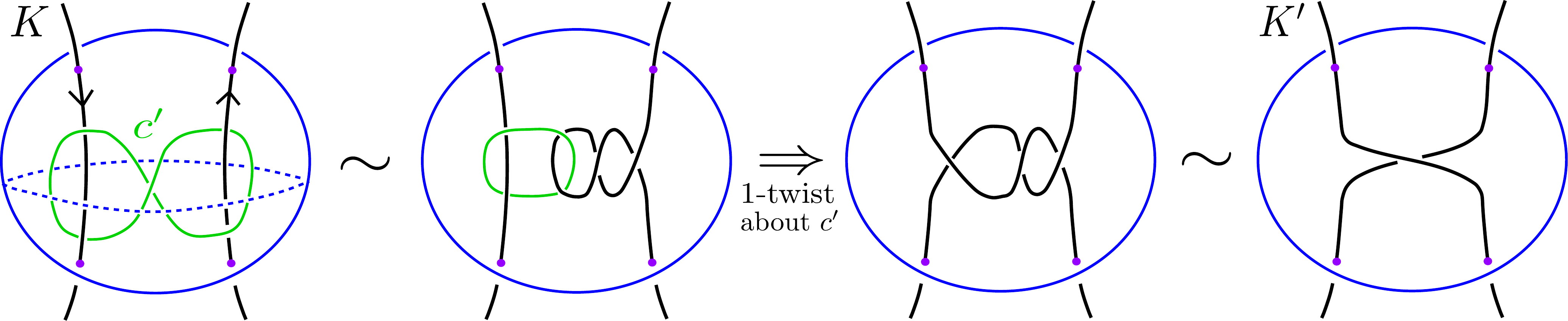}
	\caption{Realizing the effect of a two-strand $-1$-twist by instead performing a $+1$-twist along a different twisting circle.}
	\label{prop2_3_twist}
\end{figure}

At the same time, it turns out that the effect of a two-strand $-1$-twist can always be realized via a $+1$-twist on $K$ determined by another twisting circle $c'$, as shown in Figure \ref{prop2_3_twist}. In that figure, $c$ is the dashed circle sitting on the boundary of the ball in which both two-strand twists take place. While $c'$ is diagrammatically similar to $c$, it turns out that $c'$ \textit{does} bound a disk disjoint from $K$ in our example of the unknot. Figure \ref{weak_nug_unknot} provides a simple isotopy showing that $K\cup c'$ is a 2-component unlink in this case. Thus, in a sense, we may view the two-strand $-1$-twist determined by $c$ as equivalent to a nugatory $+1$-twist specified by a closely related twisting circle.

The relationship between $c'$ and $c$ can be made more precise. For one, we can view $c'$ as being embedded on the boundary of the ball $B$ in which the original two-strand $-1$-twist about $c$ takes place. Further, as shown in Figure \ref{prop2_3}$\m$, $c'$ then intersects $c$ exactly twice and meets a \textit{meridian arc} $\mu$ for $K$ in $\partial B$ exactly once, as defined at top of p. 8. Since $\mu$ lifts to a meridian of the lift $\widetilde{\gamma}$ in $\Sigma(K)$, where $\gamma$ is the original twisting arc corresponding to $c$, this means that each lift of $c'$ to $\Sigma(K)$ is a longitude of $\widetilde{\gamma}$. This implies that lifts of $c$ and $c'$ to $\Sigma(K)$ are isotopic in the manifold, although $c$ and $c'$ are distinct in $M_K$. \\

\begin{figure}[b]
	\centering
	\includegraphics[scale=0.5]{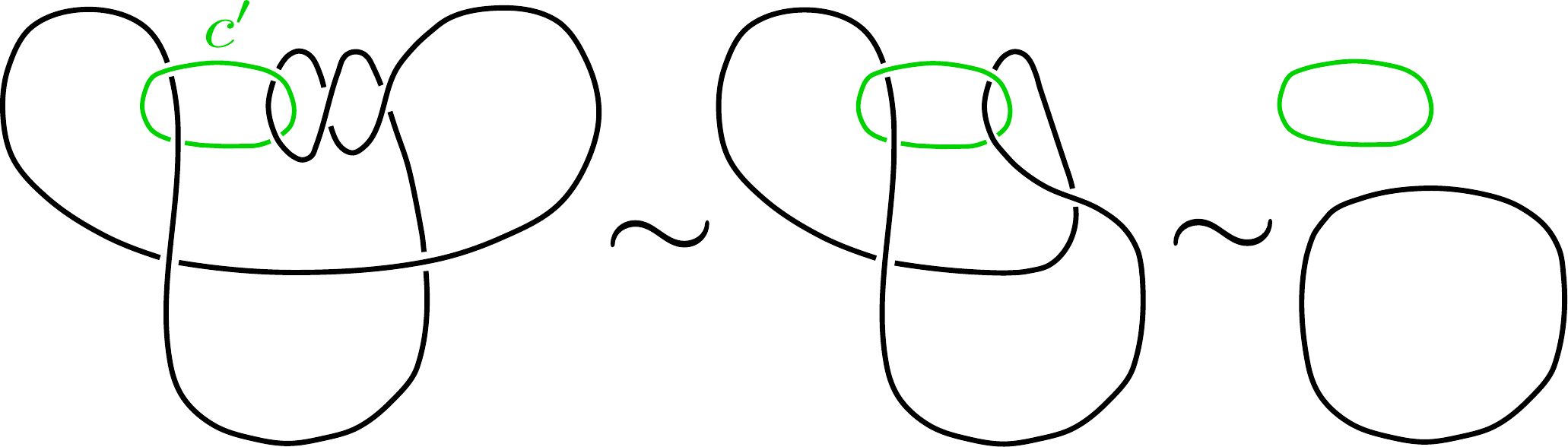}
	\caption{Isotoping the unknot $K$ from Figure \ref{cosmetic_unknot} to see that it is unlinked from the alternative twisting circle $c'$.}
	\label{weak_nug_unknot}
\end{figure}

This motivates the following definition. To set the stage, let $K$ be a knot in $M$ and $c$ be a twisting circle for $K$. Let $B$ be a 3-ball in $M$ in which a twisting disk bounded by $c$ is properly embedded, such that $\partial B$ intersects $K$ transversely in four points. The \textit{sign} of a two-strand $n$-twist naturally refers to the sign of $n$.

\begin{defn}\label{weaknug}
Let $K$, $c$, and $B$ be as above. Consider a two-strand twist of order one on $K$ in $B$, determined by $c$, which produces a knot $K'$. Say that this twist is \textit{weakly nugatory} if either it is nugatory, or there exists another twisting circle $c'$ for $K$ embedded in $\partial B-K$ such that:
\begin{description}
	\item{(a) }
	$c'$ bounds a disk properly embedded in $M_K$;
	
	\item{(b) }
	$c'$ intersects $c$ twice geometrically in $\partial B$ and intersects some meridian arc on $\partial B$ exactly once;
	
	\item{(c) }
	Performing the two-strand twist of order one on $K$ determined by $c'$, of opposite sign to that of the original twist about $c$,  produces a knot isotopic to $K'$. 
\end{description}
\end{defn}

\begin{figure}[t]
	\centering
	\includegraphics[scale=0.55]{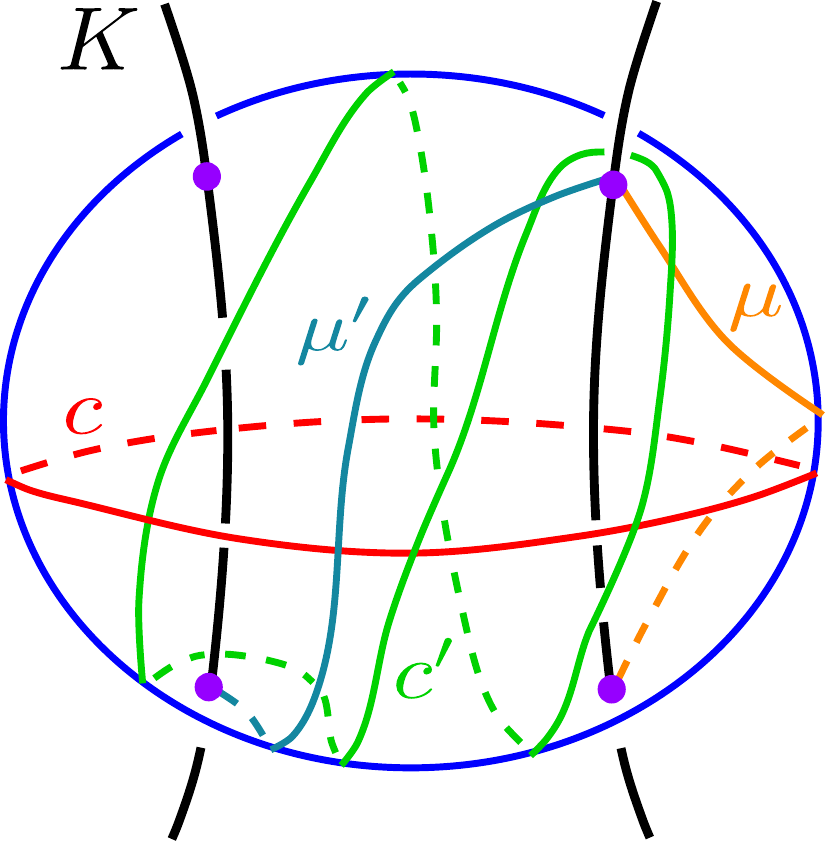}
	\caption{The curves and arcs involved in the discussion surrounding Definition \ref{weaknug}$\m$, and the proof of Proposition \ref{liftprop2}.}
	\label{prop2_3}
\end{figure}

\begin{rmk}
The reader may be wondering why we have not formulated this definition for a two-strand twist of any odd order. The reason is that there is no point: if such a definition is satisfied for a two-strand $n$-twist on $K$ with $|n|>1$, then the twist is actually nugatory. Note that, as in the example preceding this definition, condition (b) implies that a lift of $c'$ to $\Sigma(K)$ is a longitude of the lift $\widetilde{\gamma}$ of the original twisting arc. In particular, $\widetilde{\gamma}$ is then isotopic to a lift of $c'$ in $\Sigma(K)$. At the same time, condition (a) implies that each lift of $c'$ is unknotted in $\Sigma(K)$, so $\widetilde{\gamma}$ is therefore unknotted as well. By Proposition \ref{liftprop}, it follows that if $|n|>1$, then the given two-strand twist is nugatory.
\end{rmk}

The observation that $\widetilde{\gamma}$ is unknotted under the above definition is also true when $n=\pm 1$. The version of Proposition \ref{liftprop} which is valid in this case is precisely the converse of that statement.

\begin{propn}\label{liftprop2}
Suppose that $M$ is a rational homology sphere and $K$ is a null-homologous knot in $M$. Let $c$ and $\widetilde{\gamma}$ be as above, with respect to a two-strand twist of order one on $K$, and let $K'$ be the knot obtained from $K$ by performing this twist. If $K'$ is isotopic to $K$ and $\widetilde{\gamma}$ is unknotted in $\Sigma(K)$, then the twist must be weakly nugatory.
\end{propn}

\begin{proof}
We pick up from the third paragraph of the proof of Proposition \ref{liftprop} of Section 2.4: everything said up to that point makes sense for any two-strand twist determined by $c$. Recall that $N$ is the exterior of $\widetilde{\gamma}$ in $\Sigma(K)$, $\widetilde{\mu}$ is a meridian of $\widetilde{\gamma}$ on $\partial N$, $\widetilde{\mu}'$ is the slope of the Dehn filling which yields $\Sigma(K')$, and $\widetilde{\lambda}$ is the slope of the boundary of a properly embedded disk in $N$. In the case of an order one twist at hand, we have that $\widetilde{\mu}'=\widetilde{\mu}\pm\widetilde{\lambda}$ in $H_1(\partial N)$.

It follows that there are precisely two slopes on $\partial N$ which intersect both $\widetilde{\mu}$ and $\widetilde{\mu}'$ once geometrically, one of which is $\widetilde{\lambda}$. At the same time, there is a curve $c'$ on $\partial B-K$ distinct from $c$ which intersects both of the arcs $\mu$ and $\mu'$ exactly once, as shown in Figure \ref{prop2_3} for the case of a $-1$-twist. Note that $c'$ intersects $c$ exactly twice. Since $c$ and $c'$ lift to curves of distinct slopes on $\partial N$, it follows that one of $c$ and $c'$ must lift to a pair of curves of slope $\widetilde{\lambda}$.

Whichever of $c$ and $c'$ has this property must bound a properly embedded disk in $\overline{M_K-B}$. Namely, this disk will be the projection of a properly embedded disk in $N$ bounded by a copy of $\widetilde{\lambda}$. Since each disk bounded by $c'$ in $\partial B$ intersects $K$ in two points of opposite sign, as visible in Figure \ref{prop2_3}, we see that $c'$ is also a twisting circle for $K$. Consequently, to prove that the given order one twist about $c$ is weakly nugatory, it remains to show that performing the two-strand twist of order one on $K$ about $c'$, of sign opposite to that of the given twist, produces a knot isotopic to $K'$.

This fact comes from a simple local isotopy, and does not depend on whether or not $c'$ bounds a disk in $M_K$. We will demonstrate this in the case that the initial two-strand twist about $c$ is a $-1$-twist: the case of a $+1$-twist is extremely similar. Observe that, in this case, we may isotope $K$ and $c'$ together within $B$ as on the left half of Figure \ref{prop2_3_twist}. As shown in the right half of that figure, performing the two-strand $+1$-twist determined by the isotoped copy of $c'$ then yields a knot isotopic to $K'$. Since $c'$ bounds a disk embedded in $M_K$ in the case that $c$ does not, we conclude that the original order one two-strand twist from $K$ to $K'$ is weakly nugatory.
\end{proof}

This result enables us to quickly verify that, as advertised in the introduction, the unknot does not admit any cosmetic band surgeries that are more interesting than our example. As observed in Remark \ref{surgtwist}, every non-coherent band surgery on a knot can be realized (uniquely) as a two-strand twist of order one.

\begin{thm3}
Every cosmetic two-strand twist of order one on the unknot in $S^3$ is weakly nugatory.
\end{thm3}
\begin{proof}
Suppose that $\gamma$ is a twisting arc corresponding to a cosmetic two-strand $\pm 1$-twist on an unknot $K$. As in an argument made at the beginning of this section, the Gordon-Luecke theorem implies that $\widetilde{\gamma}$ is unknotted in $\Sigma(K)\cong S^3$. The conclusion then follows immediately from Proposition \ref{liftprop2}.
\end{proof}

\begin{refthm2}
In light of Proposition \ref{liftprop2}, it is natural to ask if the reasoning of Section 3 can be extended to say more about the case $n=\pm 1$ in the setting of Theorem \ref{mainthm2}. Namely, if the twisting arc hypothesis of that theorem is satisfied and a corresponding two-strand $\pm 1$-twist on $K$ produces an isotopic knot in $M$, can we argue that this twist is weakly nugatory? Unfortunately, it would not be a simple matter to extend our arguments to make this conclusion. One role that the hypothesis $|n|\geq 2$ plays in the work of Section 3 is to allow us to avoid case (a) of the conclusion of Theorem \ref{nithm} during the proof of Theorem \ref{corethm}. This case of Ni's theorem, where the surgery curve is a 1-crossing knot in $F\times[0,1]$, must be considered when the surgery slope is integral, as occurs for the surgery along $\widetilde{\gamma}$ when the order of the twist is one. 

While this situation is more complicated, the relationship between the monodromy of the fibration of $\widetilde{\Sigma}(K)$ and that of $\widetilde{\Sigma}(K')$ can still be described. By Proposition 1.4 of \cite{Ni11}, it follows readily that (up to isotopy and conjugation) these monodromies differ by composition with $T^{\pm 2}_aT^{\pm 2}_bT^{\mp 1}_c$ for a certain triple of curves $a,b,c$ which cobound a pair of pants in the fiber surface. However, since these Dehn twists are not all of the same chirality, the work of Kotschick in \cite{Kot04} cannot be used to make any conclusions about the commutator length of such a homeomorphism when $a$, $b$, and $c$ are all non-trivial.

It is not clear to the author how one might replace the application of Kotschick's work (Corollary \ref{kotcor}) in the concluding argument of Section 3 when $n=\pm 1$, in order to conclude that $\widetilde{\gamma}$ must be unknotted in $\Sigma(K)$. If this can be done, however, then Proposition \ref{liftprop} may be applied to deduce that the given $\pm 1$-twist must be weakly nugatory.
\end{refthm2}

\subsection{Questions}

Motivated by the considerations of this section, we conclude by highlighting two interesting problems for future work.

\begin{ques}\label{ques1}
Does there exist a non-trivial knot in $S^3$ admitting a cosmetic two-strand $\pm 1$-twist (i.e. band surgery) which is not weakly nugatory?
\end{ques}

\begin{ques}\label{ques2}
Given an arbitrary odd integer $n$, does there exist a knot in $S^3$ which admits a cosmetic two-strand $n$-twist? 
\end{ques}

\begin{figure}[t]
	\centering
	\includegraphics[scale=0.5]{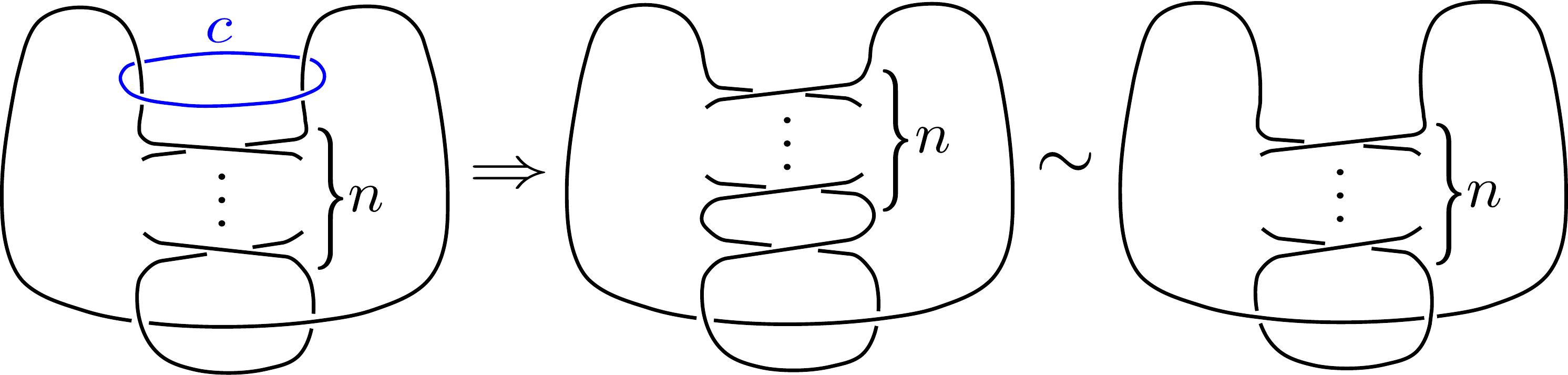}
	\caption{A chirally cosmetic two-strand $-(2n+1)$-twist on a twist knot admitting a standard diagram with $n+2$ crossings. The second operation denoted by $\sim$ is an isotopy.}
	\label{chi_cos_twist}
\end{figure}

If the answer to Question \ref{ques1} is yes, then such a cosmetic two-strand twist would provide a truly interesting example of purely cosmetic band surgery. Proposition \ref{liftprop2} may prove to be useful in answering this question negatively, at least for some classes of knots.

Of course, the examples of Section 4.1.1 answer Question \ref{ques2} affirmatively for $n=\pm 1$ and $n=\pm 5$. It is worth noting that those examples are generalized by an infinite family of \textit{chirally} cosmetic two-strand twists of odd order on twist knots, which transform them into their mirror images: see Figure \ref{chi_cos_twist}. However, since the unknot and figure-eight are the only achiral twist knots \cite{JLWW02}, this construction yields no further examples of \textit{purely} cosmetic twisting.

\bibliography{DissertationBibliography}

\bibliographystyle{siam}

\end{document}